\documentclass[12pt,reqno]{amsart}  

\usepackage{amsmath,amsthm,amssymb,amscd,mathdots,mathtools,enumerate,shuffle,stmaryrd,faktor}
\usepackage{mathabx} 
\usepackage{hyperref}

\usepackage[new]{old-arrows}
\DeclareRobustCommand\longtwoheadrightarrow
{\relbar\joinrel\twoheadrightarrow}

\usepackage{tikz,tikz-cd}
\allowdisplaybreaks
\usetikzlibrary{arrows,positioning}
\usetikzlibrary{calc}
\usetikzlibrary{decorations}
\usetikzlibrary{decorations.pathreplacing}

\newtheorem{thm}{Theorem}[section]
\newtheorem*{thm*}{Theorem}
\newtheorem*{prop*}{Proposition}
\newtheorem{prop}[thm]{Proposition}

\newtheorem{lem}[thm]{Lemma}

\newtheorem{con}[thm]{Conjecture}
\numberwithin{equation}{section}

\theoremstyle{definition}
\newtheorem{defi}[thm]{Definition}
\newtheorem*{defi*}{Definition}
\newtheorem*{con*}{Conjecture}
\newtheorem{ex}[thm]{Example}
\newtheorem{rem}[thm]{Remark}

\mathtoolsset{showonlyrefs} 

\newcommand{\QQ}{\mathbb{Q}}
\newcommand{\ZZ}{\mathbb{Z}}

\newcommand{\diff}{\mathop{}\!\mathrm{d}}
\newcommand{\nca}[2]{#1\langle #2\rangle}

\newcommand{\ncac}[2]{#1\langle \langle #2\rangle\rangle}
\newcommand{\spanQ}{\operatorname{span}_{\QQ}}
\newcommand{\Sl}{\operatorname{SL}_2(\ZZ)}
\newcommand{\quasimod}{\widetilde{\mathcal{M}}^\QQ(\Sl)}

\newcommand{\conc}{\operatorname{conc}}
\newcommand{\Hom}{\operatorname{Hom}_{\QQ\operatorname{-alg}}}
\newcommand{\id}{\operatorname{id}}
\newcommand{\anti}{\operatorname{anti}}

\newcommand{\wt}{\operatorname{wt}}
\newcommand{\dep}{\operatorname{dep}}

\newcommand{\1}{\mathbf{1}}

\newcommand{\dec}{\Delta_{\operatorname{dec}}}

\newcommand{\B}{\mathcal{B}}
\newcommand{\QB}{\nca{\QQ}{\B}}

\newcommand{\RBc}[1]{#1\langle\langle \B\rangle\rangle}

\newcommand{\Ybi}{\mathcal{Y}^{\operatorname{bi}}}
\newcommand{\QYbi}{\nca{\QQ}{\Ybi}}

\newcommand{\Y}{\mathcal{Y}}
\newcommand{\QY}{\nca{\QQ}{\Y}}

\newcommand{\RYc}[1]{#1\langle\langle \Y\rangle\rangle}

\newcommand{\X}{\mathcal{X}}
\newcommand{\QX}{\nca{\QQ}{\X}}

\newcommand{\RXc}[1]{#1\langle\langle \X\rangle\rangle}

\newcommand{\Z}{\mathcal{Z}}
\newcommand{\Zf}{\mathcal{Z}^f}
\newcommand{\co}{\Delta_{\shuffle}}
\newcommand{\stco}{\Delta_\ast}
\newcommand{\DM}{\operatorname{DM}}
\newcommand{\DMo}{\operatorname{DM}_0}

\newcommand{\dmo}{\mathfrak{dm}_0}

\newcommand{\bco}{\Delta_b}
\newcommand{\rel}{\operatorname{Rel}_{\tau,0}}
\newcommand{\BM}{\operatorname{BM}}
\newcommand{\BMo}{\operatorname{BM}_0}

\newcommand{\bmo}{\mathfrak{bm}_0}

\newcommand{\nonformal}{\mathcal{Z}_q}
\newcommand{\formal}{\mathcal{G}^{\operatorname{f}}}
\newcommand{\formalo}{\mathcal{G}^{\operatorname{f},0}}
\newcommand{\symb}{\operatorname{f}}

\newcommand{\zq}{\zeta_q}
\newcommand{\SZ}{\zeta_q^{\operatorname{SZ}}}
\newcommand{\zqreg}{\zeta_q^{\operatorname{reg}}}

\newcommand{\bst}{\ast_b}

\newcommand{\regT}{\operatorname{reg}_T}
\newcommand{\reg}{\operatorname{reg}}

\title{A generalization of formal multiple zeta values related to \\ multiple Eisenstein series and multiple $q$-zeta values}

\author{Annika Burmester}
\address{Faculty of Mathematics, Bielefeld University, Germany.}
\email{aburmester@math.uni-bielefeld.de}



\thanks{The author was partially funded by the Deutsche Forschungsgemeinschaft (DFG, German Research Foundation) -- SFB-TRR 358/1 2023 — 491392403.}

\begin{document} 
	
\maketitle

\begin{abstract} \noindent
We present the $\tau$-invariant balanced quasi-shuffle algebra $\formal$, whose elements formalize (combinatorial) multiple Eisenstein series as well as multiple $q$-zeta values. In particular, $\formal$ has natural maps into these two algebras, and we expect these maps to be isomorphisms. 
Racinet studied the algebra $\Zf$ of formal multiple zeta values by examining the corresponding affine scheme $\DM$. Similarly, we present the affine scheme $\BM$ corresponding to the algebra $\formal$. 
We show that Racinet's affine scheme $\DM$ embeds into our affine scheme $\BM$. This leads to a projection from the algebra $\formal$ onto $\Zf$. Via the above natural maps, this projection corresponds to extracting the constant terms of multiple Eisenstein series or the limit $q\to1$ of multiple $q$-zeta values.
\end{abstract} 


\section{Introduction} 

\noindent
\emph{Multiple zeta values} are real numbers defined for integers $k_1\geq2,\ k_2,\ldots,k_d\geq1$ by
\[\zeta(k_1,\ldots,k_d)=\sum_{n_1>\dots > n_d>0} \frac{1}{n_1^{k_1}\cdots n_d^{k_d}}.\]
We refer to the number $k_1+\dots+k_d$ as the \emph{weight} and to the number $d$ as the \emph{depth}. Multiple zeta values were first introduced in depth two by C. Goldbach and L. Euler more than two centuries ago. In recent years these values were studied intensively due to their rich structure and their occurrence in various areas of mathematics, such as number theory, algebraic geometry, knot theory, quantum field theory, and also in mathematical physics. A survey on achievements in the theory of multiple zeta values can be found in \cite{bgf} and \cite{zh}, and all articles related to multiple zeta values are listed in \cite{hw}. \\
The product of multiple zeta values can be expressed in two different ways; one is called the \emph{stuffle product} and comes from the definition of multiple zeta values as infinite nested sums, and the other one is called the \emph{shuffle product} and comes from the representation of multiple zeta values as iterated integrals. Both products possess a description in terms of quasi-shuffle algebras (introduced in \cite{h}, \cite{hi}). Comparing these two product formulas together with some regularization process (given in \cite{ikz}) yields the \emph{extended double shuffle relations} among multiple zeta values. A main conjecture in the theory of multiple zeta values is that the extended double shuffle relations give all algebraic relations in the algebra 
\[\Z=\spanQ\{\zeta(k_1,\ldots,k_d)\mid d\geq0, k_1\geq2,k_2,\ldots,k_d\geq1\}. \] 
This motivates to consider the algebra $\Zf$ of \emph{formal multiple zeta values} spanned by the elements $\zeta^f(k_1,\ldots,k_d)$, $k_1\geq2,k_2,\ldots,k_d\geq1$, which exactly satisfy the extended double shuffle relations. By construction, there is a realization of $\Zf$ into $\Z$. In other words, there is a surjective algebra morphism
\begin{align*}
\Zf&\longtwoheadrightarrow\Z,\\
\zeta^f(k_1,\ldots,k_d)&\longmapsto\zeta(k_1,\ldots,k_d),
\end{align*} 
which should be an isomorphism by the previously mentioned main conjecture.
\vspace{0,2cm} \\
In his thesis (\cite{rac}), Racinet studied the algebraic structure of formal multiple zeta values by considering an affine scheme $\DMo$ represented by the quotient algebra $\Zf/(\zeta^f(2))$. This means, for each commutative $\QQ$-algebra $R$ with unit we have natural bijections
\begin{align} \label{DM represented by intro}
\Hom\Big(\faktor{\Zf\!}{\!(\zeta^f(2))},R\Big)\simeq \DMo(R).
\end{align}
Then, he showed that $\DMo$ is actually an affine group scheme. This yields an isomorphism of algebras
\begin{align*}
\Zf\simeq \QQ[\zeta^f(2)]\otimes_\QQ \mathcal{U}(\dmo)^\vee,
\end{align*}
where $\dmo$ is the corresponding Lie algebra to $\DMo$.
In particular, $\Zf$ is a free polynomial algebra. The Lie bracket on $\dmo$ induces a coproduct on the quotient $\Zf/(\zeta^f(2))$, which is usually referred to as Goncharov's coproduct (\cite{gon}). This coproduct equips $\Zf/(\zeta^f(2))$ with a weight-graded Hopf algebra structure. The purpose of this article is to advocate a generalization of \eqref{DM represented by intro}.
\vspace{0,2cm} \\
A kind of deformation of multiple zeta values are \emph{$q$-analogs of multiple zeta values}; these are elements in $\QQ\llbracket q\rrbracket$ degenerating to multiple zeta values for a suitable limit $q\to 1$. To get nice algebraic structures, we have to restrict to a particular kind of $q$-analog of multiple zeta values. An example for such $q$-analogs are the \emph{Schlesinger-Zudilin multiple $q$-zeta values} first studied in \cite{schl}, \cite{zu}, and then in their extended version in \cite{ems}. These are defined for integers $s_1\geq1, s_2,\ldots,s_l\geq0$ as
\begin{align*}
\SZ(s_1,\ldots,s_l)=\sum_{n_1>\dots > n_l>0} \frac{q^{n_1s_1}}{(1-q^{n_1})^{s_1}}\cdots \frac{q^{n_ls_l}}{(1-q^{n_l})^{s_l}}\in \QQ\llbracket q \rrbracket.
\end{align*}
We call the number $s_1+\dots+s_l+\#\{i\mid s_i=0\}$ the \emph{weight} and the number $l-\#\{i\mid s_i=0\}$ the \emph{depth}. Denote
\begin{align*}
\nonformal=\spanQ\{\SZ(s_1,\ldots,s_l)\mid l\geq0,\ s_1\geq1,\ s_2,\ldots,s_l\geq0\},
\end{align*}
where $\SZ(\emptyset)=1$. Since the Schlesinger-Zudilin multiple $q$-zeta values are defined as infinite nested sums, they also satisfy a stuffle product. For example, we have for $s_1,s_2\geq1$
\begin{align} \label{SZ stuffle intro}
\SZ(s_1)\SZ(s_2)=\SZ(s_1,s_2)+\SZ(s_2,s_1)+\SZ(s_1+s_2).
\end{align}
In particular, the space $\nonformal$ is an algebra. The Schlesinger-Zudilin multiple $q$-zeta values are invariant under a certain involution $\tau$ (see \eqref{SZ tau invariance}), which is closely connected to conjugation of partitions. Precisely, we have for  $k_1,\ldots,k_d\geq1,\ m_1,\ldots,m_d\geq0$
\begin{align} \label{SZ tau invariance intro}
\SZ(k_1,\{0\}^{m_1},\ldots,k_d,\{0\}^{m_d})=\SZ(m_d+1,\{0\}^{k_d-1},\ldots,m_1+1,\{0\}^{k_1-1}).
\end{align}
\begin{con*} All algebraic relations in $\nonformal$ are a consequence of the stuffle product formula \eqref{SZ stuffle intro} and the $\tau$-invariance \eqref{SZ tau invariance intro} of the Schlesinger-Zudilin multiple $q$-zeta values. 
\end{con*} \noindent
This conjecture was first stated by Bachmann in the context of bi-brackets. One easily verifies that the $\tau$-invariance in \eqref{SZ tau invariance intro} is homogeneous in weight, but the stuffle product in \eqref{SZ stuffle intro} is not. In analogy to the multiple zeta values and to apply similar techniques as in Racinet's approach to formal multiple zeta values, we are more interested in weight-homogeneous relations. Precisely, we want to replace the stuffle product of the Schlesinger-Zudilin multiple $q$-zeta values with the associated weight-graded product. In \cite{bu2}, a spanning set of $\nonformal$ is constructed, which indeed satisfies the associated weight-graded product and the $\tau$-invariance of the Schlesinger-Zudilin multiple $q$-zeta values. The elements of this spanning set are called \emph{balanced multiple $q$-zeta values}. 

\noindent \vspace{-0,2cm} \\ 
Another family of objects related to multiple zeta values are the \emph{multiple Eisenstein series} introduced in \cite{gkz}. For integers $k_1\geq3, k_2,\ldots,k_d\geq2$ these are given by
\[\mathbb{G}_{k_1,\ldots,k_d}(z)=\sum_{\substack{\lambda_1\succ\cdots\succ \lambda_d\succ 0 \\ \lambda_i\in \ZZ z+\ZZ}} \frac{1}{\lambda_1^{k_1}\cdots \lambda_d^{k_d}} \qquad (z\in \mathbb{H}).\]
Here we set $m_1z+n_1\succ m_2z+n_2$ if $m_1>m_2$ or $(m_1=m_2 \wedge n_1>n_2)$. Every multiple Eisenstein series $\mathbb{G}_{k_1,\ldots,k_d}(z)$ possesses a Fourier expansion, where the constant term is exactly the multiple zeta value $\zeta(k_1,\ldots,k_d)$. More precisely, the building blocks of the Fourier expansion are multiple zeta values and so-called \emph{brackets} studied in \cite{bk2}. The brackets are contained in $\nonformal$; thus with the identification $q=e^{2\pi i z}$ we have
\begin{align*}
\mathbb{G}_{k_1,\ldots,k_d}(z)\in \nonformal \otimes \Z[\pi i].
\end{align*}
There exist stuffle and shuffle regularized multiple zeta values for all positive multi indices $(k_1,\ldots,k_d)\in \mathbb{N}^d$, and an extension of the brackets, called \emph{bi-brackets} (see \cite{ba}). These two extended families of objects allow to define stuffle and shuffle regularized multiple Eisenstein series $\mathbb{G}_{k_1,\ldots,k_d}(z)$ for all $k_1,\ldots,k_d\geq 1$ (see \cite{ba3}, \cite{bt}). Computer experiments by Bachmann and Tasaka in \cite{bt} suggest that the relations between (regularized) multiple Eisenstein series are equal to the relations between brackets modulo lower weight. \\
To describe the relations among the brackets one needs the bi-brackets. It is conjectured by Bachmann that brackets and bi-brackets actually span the same space, namely $\nonformal$. 
In \cite{bb} a spanning set of $\nonformal$ is constructed which conjecturally satisfies exactly the relations of the bi-brackets modulo lower weight. Those are called \emph{combinatorial bi-multiple Eisenstein series} 
\[G\binom{k_1,\ldots,k_d}{m_1,\ldots.m_d}\in \nonformal, \qquad k_1,\ldots,k_d\geq1,\ m_1,\ldots,m_d\geq0.\] 
In summary, the combinatorial bi-multiple Eisenstein series should explain all algebraic relations between multiple Eisenstein series.
\vspace{0,3cm} \\
\textbf{Main results.} Consider the alphabet $\B=\{b_0,b_1,b_2,\ldots\}$ and let $\bst$ be the quasi-shuffle product on the non-commutative free algebra $\QB$ recursively given by $\1\bst w=w\bst\1=w$ and
\begin{align*}
b_iu\bst b_jv=b_i(u\bst b_jv)+b_j(b_iu\bst v)+\begin{cases} b_{i+j}(u\bst v),&\quad i,j\geq1, \\ 0 &\quad \text{else}, \end{cases}
\end{align*}
for all $u,v,w\in \QB$. We refer to $\bst$ as the \emph{balanced quasi-shuffle product}. 
On the subspace \[\QB^0=\QB\backslash b_0\QB,\] 
we define the involution $\tau$ by
\begin{align*}
\tau(b_{k_1}b_0^{m_1}\cdots b_{k_d}b_0^{m_d})=b_{m_d+1}b_0^{k_d-1}\cdots b_{m_1+1}b_0^{k_1-1}
\end{align*}
for all $k_1,\ldots,k_d\geq1,\ m_1,\ldots,m_d\geq0$. 
\begin{defi*} We set
\begin{align*}
\formal=\faktor{(\QB,\bst)}{\rel},
\end{align*}
where $\rel$ is the ideal generated by the elements $b_0$ and $\tau(w)-w$ for each $w\in \QB^0$. We denote the class of an element $w\in \QB$ by $\symb(w)$. A \emph{realization} of $\formal$ is a surjective algebra morphism $\formal\to R$ into some $\QQ$-algebra $R$.
\end{defi*} 
\begin{thm*}[Theorem \ref{realization of formal}] There is a realization of $\formal$ into the algebra $\nonformal$ given by
\begin{align*}
\formal&\longtwoheadrightarrow \nonformal, \\
\symb(w)&\longmapsto\zqreg(w).
\end{align*}
\end{thm*} \noindent
The elements $\zqreg(w)$ are the \emph{regularized balanced multiple $q$-zeta values}. We expect the realization $\formal\to\nonformal$ in Theorem \ref{realization of formal} to be an isomorphism, meaning that all algebraic relations in $\nonformal$ come from the balanced quasi-shuffle product formula and the $\tau$-invariance of the balanced multiple $q$-zeta values.\\
Denote by $\formalo$ the subalgebra generated by the elements $\symb(w)$, where $w\in \QB$ is a word consisting of the letters $b_i$, $i\geq1$. Then, a candidate for a realization into the algebra $\mathcal{E}$ of multiple Eisenstein series is given by
\begin{align*}
\formalo &\longtwoheadrightarrow \mathcal{E}, \\
\symb(b_{k_1}\cdots b_{k_d})&\longmapsto\mathbb{G}_{k_1,\ldots,k_d}(z).
\end{align*}
Here the images are the stuffle regularized multiple Eisenstein series. By the previously mentioned conjecture of Bachmann (see also Conjecture \ref{conj G^{f,0} equals G^f}), the algebra $\formalo$ should coincide with $\formal$.
\vspace{0,2cm} \\
Inspired by Racinet's ideas for formal multiple zeta values (\cite{rac}), we introduce an affine scheme $\BM:\QQ\operatorname{-Alg}\to\operatorname{Sets}$ having values in the complete Hopf algebra $(\ncac{R}{\B},\conc,\bco)$ (given in Proposition \ref{dual Hopf algebra}), and show the following.
\begin{thm*}[Theorem \ref{BM affine scheme}] The affine scheme $\BM$ is represented by the algebra $\formal$, i.e., there are natural bijections
\[\Hom(\formal,R)\simeq \BM(R)\]
for each commutative $\QQ$-algebra $R$ with unit.
\end{thm*} \noindent
The affine scheme $\BM$ generalizes Racinet's affine scheme $\DM$. Precisely, we have an injective morphism of affine schemes
\begin{align*}
\theta: \DM\longhookrightarrow \BM.
\end{align*}
Applying Yoneda's Lemma to the morphism $\theta$ gives the following main result.
\begin{thm*}[Theorem \ref{surjection G^f to Z^f}] There is a surjective algebra morphism
\begin{align*}
p:\formal\longtwoheadrightarrow \Zf.
\end{align*}
\end{thm*} \noindent
In other words, there is a realization of the algebra $\formal$ into the algebra $\Z$ of multiple zeta values. The morphism $p$ should be seen as a formal version of the limit $q\to1$ (as computed in \cite[Proposition 11.5, Remark 11.6]{bu2}) or, equivalently, as a formal version of taking the constant term of multiple Eisenstein series.\\
The map $p:\formal\to \Zf$ allows us to recover the extended double shuffle relations in $\Zf$ from the relations satisfied in the algebra $\formal$. So we obtain a new point of view on the extended double shuffle relations, which might give a possibility to make some progress towards the main conjecture for multiple zeta values (Conjecture \ref{Conj all relations in Z}). \\
On the other hand, Theorem \ref{surjection G^f to Z^f} gives evidence that there should be an approach to the algebra $\formal$ similar to the one given by Racinet for the algebra $\Zf$ of formal multiple zeta values (see Section \ref{MZV}). In particular, this paper gives the first step towards the expectation that $\formal$ is a free polynomial algebra and is equipped with a Hopf algebra structure, where the coproduct is a generalization of Goncharov's coproduct.

\noindent \vspace{-0,5cm} \\
\paragraph{\textbf{Structure of the paper.}} In Section \ref{MZV} we briefly recall (formal) multiple zeta values and illustrate Racinet's approach to them via the affine group scheme $\DMo$ and the Lie algebra $\dmo$. In Section \ref{nonformal} we describe the algebra $\nonformal$ and its relation to $q$-analogs of multiple zeta values, and multiple Eisenstein series. Then in Section \ref{formal} we define the algebra $\formal$ and explain the realization into $\nonformal$ given by the balanced multiple $q$-zeta values. In Section \ref{balanced quasi-shuffle Hopf algebra}, we introduce the balanced quasi-shuffle Hopf algebra and its completed dual, which will be necessary for defining the affine scheme $\BM$.  Then in Section \ref{affine scheme BM}, we give the affine scheme $\BM$, which is represented by the algebra $\formal$. Finally, in Section \ref{BM and DM} we prove that we have an embedding $\DM\hookrightarrow\BM$. This leads to the surjective algebra morphism $\formal \twoheadrightarrow \Zf$.

\noindent \vspace{-0,5cm} \\
\paragraph{\textbf{Acknowledgment.}} This work was part of my PhD thesis (\cite{bu}). Therefore, I deeply thank my supervisor Ulf Kühn for many helpful comments and discussions on these contents. Furthermore, I would like to thank Claudia Alfes-Neumann, Henrik Bachmann, Jan-Willem van Ittersum, and Koji Tasaka for valuable comments on these topics within my PhD project and on earlier versions of this paper.

\section{Racinet's approach to multiple zeta values} \label{MZV}

\noindent
We recall Racinet's approach to the algebra of formal multiple zeta values via affine group schemes and Lie algebras. In particular, we associate an affine group scheme $\DMo$ to the algebra of formal multiple zeta values having values in a completed Hopf algebra. Then, we briefly illustrate the structural results for the algebra of formal multiple zeta values obtained from this approach. This will motivate our approach to the algebra $\formal$ in the following sections. We start with a short basic introduction to the theory of multiple zeta values, for details we refer to \cite[Chapter 1]{bgf}.
\begin{defi}  \label{def MZV}
To integers $k_1\geq 2,\ k_2,\ldots,k_d\geq 1$, associate the \emph{multiple zeta value}
\[\zeta(k_1,\ldots,k_d)=\sum_{n_1>\dots>n_d>0} \frac{1}{n_1^{k_1}\cdots n_d^{k_d}} \in \mathbb{R}.\]
The number $k_1+\dots+k_d$ is called the \emph{weight} and the number $d$ is called the \emph{depth}. Denote the $\QQ$-vector space spanned by all multiple zeta values by
\[\Z=\spanQ\{\zeta(k_1,\ldots,k_d)\mid d\geq 0,\ k_1\geq2,\ k_2,\ldots,k_d\geq 1\},\]
where $\zeta(\emptyset)=1$.
\end{defi} \noindent
There are two ways of expressing the product of multiple zeta values; both of them have a description in terms of quasi-shuffle algebras (\cite{h}, \cite{hi}). \\
Consider the alphabet $\Y=\{y_1,y_2,y_3,\ldots\}$. Let $\Y^*$ be the set of all words with letters in $\Y$, and denote by $\QY$ the free non-commutative polynomial algebra over $\QQ$ generated by $\Y$. We denote the empty word by $\1$. Define the \emph{stuffle product} $\ast$ on $\QY$ recursively by $\1\ast w=w\ast\1=w$ and
\begin{align*}
y_iu\ast y_jv=y_i(u\ast y_jv)+y_j(y_iu\ast v)+y_{i+j}(u\ast v)
\end{align*}
for all $u,v,w\in \QY$. Then there is a surjective algebra morphism
\begin{align} \label{stuffle mzv}
(\QY,\ast)&\longrightarrow \Z, \\
y_{k_1}\cdots y_{k_d}&\longmapsto \zeta_{\ast}(k_1,\ldots,k_d).
\end{align}
The elements $\zeta_{\ast}(k_1,\ldots,k_d)$ are the \emph{stuffle regularized multiple zeta values}; they are uniquely determined by $\zeta_{\ast}(k_1,\ldots,k_d)=\zeta(k_1,\ldots,k_d)$ for $k_1\geq2$, $k_2,\ldots,k_d\geq1$ and $\zeta_{\ast}(1)=0$.
\\
Consider the alphabet $\X=\{x_0,x_1\}$. Denote by $\X^*$ the set of all words with letters in $\X$, and let $\QX$ be the free non-commutative polynomial algebra over $\QQ$ generated by $\X$. The \emph{shuffle product} on $\QX$ is recursively defined by $\1\shuffle w=w\shuffle \1=w$ and 
\begin{align*}
x_iu\shuffle x_jv=x_i(u\shuffle x_jv)+x_j(x_iu\shuffle v)
\end{align*}
for all $u,v,w\in \QX$. Using the iterated integral expression of the multiple zeta values, one obtains a surjective algebra morphism
\begin{align} \label{shuffle mzv}
(\QX,\shuffle)&\longrightarrow \Z, \\
x_{\varepsilon_1}\cdots x_{\varepsilon_n}&\longmapsto \zeta_{\shuffle}(x_{\varepsilon_1}\cdots x_{\varepsilon_n}).
\end{align}
The elements $\zeta_{\shuffle}(x_{\varepsilon_1}\cdots x_{\varepsilon_n})$ are the \emph{shuffle regularized multiple zeta values}; they are uniquely determined by $\zeta_{\shuffle}(x_0^{k_1-1}x_1\cdots x_0^{k_d-1}x_1)=\zeta(k_1,\ldots,k_d)$ for all $k_1\geq2$, $k_2,\ldots,k_d\geq1$ and $\zeta_{\shuffle}(x_0)=\zeta_{\shuffle}(x_1)=0$.
\\
There is an explicit connection between the shuffle and stuffle regularized multiple zeta values (see \cite{ikz}) given by
\begin{align} \label{comparison stuffle and shuffle regularized}
\sum_{w\in \Y^*} \zeta_{\ast}(w) w=\exp\left(\sum_{n\geq2} \frac{(-1)^{n-1}}{n}\zeta(n)y_1^n\right)\sum_{w\in \X^*} \zeta_{\shuffle}(w) \Pi_{\Y}(w).
\end{align}
Here the map $\Pi_{\Y}:\QX\to \QY$ is the canonical projection sending each word ending with $x_0$ to $0$ and $x_0^{k_1-1}x_1\cdots x_0^{k_d-1}x_1$ to $y_{k_1}\cdots y_{k_d}$. Comparing the stuffle product formula for the stuffle regularized multiple zeta values \eqref{stuffle mzv} and the shuffle product formula for the shuffle regularized multiple zeta values \eqref{shuffle mzv} by using \eqref{comparison stuffle and shuffle regularized} gives the \emph{extended double shuffle relations} among multiple zeta values.
\begin{con} \label{Conj all relations in Z} (\cite[Conjecture 1]{ikz}) All algebraic relations in the algebra $\Z$ of multiple zeta values are a consequence of the extended double shuffle relations.
\end{con} \noindent
In particular, Conjecture \ref{Conj all relations in Z} would imply that the algebra $\Z$ is graded by weight, since the stuffle and shuffle product are both homogeneous in weight. Conjecture \ref{Conj all relations in Z} is a motivation to study the following.
\begin{defi} \label{algebra Z^f} 
The algebra of \emph{formal multiple zeta values} is 
\[\Zf=\faktor{(\QX,\shuffle)}{\operatorname{Rel}_{\operatorname{EDS}}},\]
where $\operatorname{Rel}_{\operatorname{EDS}}$ is the ideal in $(\QX,\shuffle)$ generated by the extended double shuffle relations.
\end{defi} \noindent 
We denote the class of an element $w\in \QX$ in the quotient $\Zf$ by $\zeta_{\shuffle}^f(w)$ and set $\zeta_{\shuffle}^f(\1)=1$. If $w\in \QX$ is a $\QQ$-linear combination of words starting with $x_0$ and ending with $x_1$, we abbreviate $\zeta^f(w)=\zeta_{\shuffle}^f(w)$. By construction, there is a realization of the algebra $\Zf$ into $\Z$ given by
\begin{align*}
\Zf&\longtwoheadrightarrow\Z, \\
\zeta_{\shuffle}^f(w)&\longmapsto\zeta_{\shuffle}(w).
\end{align*}
As a reformulation of Conjecture \ref{Conj all relations in Z}, this morphism is expected to be an isomorphism.
\vspace{0,2cm} \\
To associate an affine group scheme to the algebra $\Zf$ of formal multiple zeta values, we have to dualize the shuffle and stuffle product. Let $\dec$ be the deconcatenation coproduct, i.e., for each word $w$ in $\QX$ or $\QY$ we have
\begin{align*}
\dec(w)=\sum_{uv=w} u\otimes v.
\end{align*}
Then $(\QX,\shuffle,\dec)$ and $(\QY,\ast,\dec)$ are both weight-graded commutative Hopf algebras. Let $R$ be some commutative $\QQ$-algebra with unit and denote by $\RXc{R}$ resp. $\RYc{R}$ the free algebra of non-commutative power series in the alphabet $\X$ resp. $\Y$ with coefficients in $R$. The dual completed, cocommutative Hopf algebra to $(\QX,\shuffle,\dec)$ is given by $(\RXc{R},\conc,\co)$, where $\conc$ denotes the usual concatenation product and $\co$ is given on the generators by
\begin{align*}
\co(x_i)=x_i\otimes\1+\1\otimes x_i, \qquad i\in\{0,1\}.
\end{align*}
The duality pairing is
\begin{align*}
\RXc{R}\times \QX&\longrightarrow R, \\
(f,w)&\longmapsto (f\mid w),
\end{align*} 
where $(f\mid w)$ denotes the coefficient of $w$ in $f$. So, an element $f\in \RXc{R}$ is grouplike for $\co$, i.e., we have $\co(f)=f\otimes f$, if and only if $(f\mid u\shuffle v)=(f\mid u)(f\mid v)$ for all $u,v\in \QX$.\\ 
Similarly, the dual completed, cocommutative Hopf algebra to $(\QY,\ast,\dec)$ is given by $(\RYc{R},\conc,\stco)$, where the coproduct $\stco$ is defined on the generators by
\begin{align*}
\stco(y_i)=y_i\otimes\1+\1\otimes y_i+\sum_{j=1}^{i-1} y_j\otimes y_{i-j}, \qquad i\geq1.
\end{align*}
An element $f\in \RYc{R}$ is grouplike for $\stco$ if and only if $(f\mid u\ast v)=(f\mid u)(f\mid v)$ for all $u,v\in \QY$.
\begin{defi} \label{def DM} (\cite{rac}) For any commutative $\QQ$-algebra $R$ with unit, let $\DM(R)$ be the set of all non-commutative power series $\phi\in \RXc{R}$ satisfying
\renewcommand{\arraystretch}{1,1}
\begin{center} \centering \begin{tabular}{cclc}
(i) & $(\phi|x_0)\ =\ (\phi|x_1)$ & $=$ & $0$, \\ 
(ii) & $\co(\phi)$ & $=$ & $\phi \otimes \phi$, \\
(iii) & $\stco(\phi_\ast)$ & $=$ & $\phi_\ast \otimes \phi_\ast$, 
\end{tabular} \end{center} \renewcommand{\arraystretch}{1}
where 
\[\phi_\ast=\exp\left(\sum\limits_{n\geq 2}\frac{(-1)^{n-1}}{n}(\Pi_\Y(\phi)|y_n)y_1^n\right)\Pi_\Y(\phi)\in \RYc{R},\] 
and $\Pi_\Y$ is the $R$-linear, completed extension of the canonical projection $\QX\to\QY$ sending each word ending with $x_0$ to $0$ and $x_0^{k_1-1}x_1\cdots x_0^{k_d-1}x_1$ to $y_{k_1}\cdots y_{k_d}$.
\vspace{0,2cm} \\
For each $\lambda\in R$, denote by $\DM_\lambda(R)$ the set of all $\phi\in \DM(R)$, which additionally satisfy
\[\hspace{-2,7cm}\text{(iv)} \qquad (\phi|x_0x_1) = \lambda. \]
\end{defi} \noindent
For example, we have for the non-commutative generating series of the shuffle regularized multiple zeta values
\[\sum_{w\in \X^*} \zeta_{\shuffle}(w)w\in \DM_{\pi^2/6}(\Z).\]
By \cite[Theorem I]{rac2}, the set $\DM_\lambda(R)$ is always non-empty.
\vspace{0,2cm} \\
In the following, we let $\QQ\operatorname{-Alg}$ be the category of commutative $\QQ$-algebras with unit and denote by $\operatorname{Sets}$ the category of all sets.
\begin{thm} \label{DM affine scheme} 
(i) The functor $\DM:\QQ\operatorname{-Alg}\to \operatorname{Sets}$ is an affine scheme represented by the algebra $\Zf$ of formal multiple zeta values. \\
(ii) The functor $\DM_\lambda:\QQ\operatorname{-Alg}\to \operatorname{Sets}$ is an affine scheme represented by the quotient algebra \[\faktor{\Zf\!}{\!(\zeta^f(2)-\lambda)}.\]
\end{thm} \noindent
More precisely, for each $R\in\QQ\operatorname{-Alg}$ there are bijections 
\begin{align} \label{bijections DM}
\hspace{-0,4cm}\Hom(\Zf,R)\overset{\sim}{\longrightarrow}\DM(R), \quad \Hom\Big(\faktor{\Zf\!}{\!(\zeta^f(2)-\lambda)}, R\Big)\overset{\sim}{\longrightarrow}\DM_\lambda(R).
\end{align}
In the case $\lambda=0$, the functor $\DMo$ is an affine group scheme (\cite[Section IV]{rac}). The group multiplication on $\DMo(R)$ is given by 
\begin{align*} 
G\circledast H=G\kappa_G(H), \qquad \quad \text{for} \quad G,H\in \DMo(R),
\end{align*}
where $\kappa_{G}$ is the algebra automorphism on $\RXc{R}$ defined by $\kappa_G(\1)=\1$, $\kappa_G(x_0)=x_0$ and $\kappa_G(x_1)=G^{-1}x_1G$. To prove this Racinet considered the linearized space $\dmo$ and showed that it is Lie algebra with the Ihara bracket, which is exactly the linearized operation to $\circledast$. After showing $\exp_{\circledast}(\dmo)=\DMo$, this implies that $\DMo$ is an affine group scheme. A consequence of this story is the following theorem first stated by J. Ecalle.
\begin{thm} \label{Z^f polynomial algebra} (\cite[Section IV, Corollary 3.14]{rac}) There is an algebra isomorphism
\[ \Zf\simeq \QQ[\zeta^f(2)]\otimes_\QQ\mathcal{U}(\dmo)^\vee.\]
So, $\Zf$ is a free polynomial algebra.
\end{thm} \noindent
By dualizing the product $\circledast$ on $\DMo$ one obtains a coproduct, which is usually referred to as Goncharov's coproduct (\cite{gon}). A second structural consequence for $\Zf$ is obtained from combining \cite[Section IV]{rac} and the calculations in \cite[Proposition 3.418]{bgf}.
\begin{thm} The algebra $\faktor{\Zf\!}{\!(\zeta^f(2))}$ equipped with Goncharov's coproduct is a weight-graded Hopf algebra.
\end{thm}

\section{The algebra $\nonformal$} \label{nonformal}

\noindent
We present the algebra $\nonformal$, which combines $q$-analogs of multiple zeta values, (combinatorial) multiple Eisenstein series, and also generating series of certain functions on partitions. We illustrate the strong connections to the algebra $\Z$ of multiple zeta values.

\subsection{$q$-analogs of multiple zeta values.}
A $q$-analog of some expression is a generalization involving the variable $q$, which returns the original expression for the limit $q\to 1$. For example, a $q$-analog of some positive integer $n$ is given by
\begin{align*}
[n]_q=\frac{1-q^n}{1-q}=1+q+q^2+\dots+q^{n-1}.
\end{align*}
We are interested in $q$-analogs of multiple zeta values.
\begin{defi} (\cite{schl}, \cite{zu}, \cite{ems}) For integers $s_1\geq1,\ s_2,\ldots,s_l\geq0$, the \emph{Schlesinger-Zudilin multiple $q$-zeta value} is
\[\SZ(s_1,\ldots,s_l)=\sum_{n_1>\dots>n_l>0} \frac{q^{n_1s_1}}{(1-q^{n_1})^{s_1}}\cdots \frac{q^{n_ls_l}}{(1-q^{n_l})^{s_l}}\in \QQ\llbracket q\rrbracket.\]
The number $s_1+\dots+s_l+\#\{i\mid s_i=0\}$ is called its \emph{weight} and the number $l-\#\{i\mid s_i=0\}$ its \emph{depth}.
\end{defi} \noindent
For integers $s_1\geq2, s_2,\ldots,s_l\geq1$, one easily verifies
\begin{align*}
\lim_{q\to 1} (1-q)^{s_1+\dots+s_l}\SZ(s_1,\ldots,s_l)=\lim_{q\to1} \sum_{n_1>\dots>n_l>0} \frac{q^{n_1s_1}}{[n_1]_q^{s_1}}\cdots \frac{q^{n_ls_l}}{[n_l]_q^{s_l}}=\zeta(s_1,\ldots,s_l).
\end{align*}
So the Schlesinger-Zudilin multiple $q$-zeta values are (sometimes called modified) $q$-analogs of multiple zeta values. More generally, in \cite{bi} it is shown that the regularized limit of all Schlesinger-Zudilin multiple $q$-zeta values lies in $\Z$.
\begin{defi} We denote
\[\nonformal=\spanQ\{\SZ(s_1,\ldots,s_l)\mid l\geq0,\ s_1\geq1,\ s_2,\ldots,s_l\geq0\},\]
where $\SZ(\emptyset)=1$.
\end{defi} \noindent
The space $\nonformal$ is an algebra. To describe the product of the Schlesinger-Zudilin multiple $q$-zeta values, one usually uses quasi-shuffle products. 
\\
Consider the alphabet $\B=\{b_0,b_1,b_2,\ldots\}$. Let $\B^*$ the set of all words with letters in $\B$, and denote by $\QB$ the free non-commutative polynomial algebra over $\QQ$ generated by $\B$. We define the \emph{weight} and \emph{depth} of a word in $\QB$ as
\begin{align*}
\wt(b_{s_1}\cdots b_{s_l})=s_1+\dots+s_l+\#\{i\mid s_i=0\},\qquad \dep(b_{s_1}\cdots b_{s_l})=l-\#\{i\mid s_i=0\}.
\end{align*}
The SZ-stuffle product on $\QB$ is recursively defined by $\1\ast_{\operatorname{SZ}} w=w\ast_{\operatorname{SZ}} \1=w$ and 
\begin{align*}
b_iu\ast_{\operatorname{SZ}}b_jv=b_i(u\ast_{\operatorname{SZ}}b_jv)+b_j(b_iu\ast_{\operatorname{SZ}}v)+b_{i+j}(u\ast_{\operatorname{SZ}}v)
\end{align*}
for all $u,v,w\in \QB$. It is filtered with respect to the weight and the depth, but not graded. In the SZ-stuffle product always occur terms of lower depth, and if one of the factors contains the letter $b_0$ then in the SZ-stuffle product also words of lower weight appear. Let
\begin{align*}
\QB^0=\QB\backslash b_0\QB
\end{align*}
be the subspace spanned by all words not starting in $b_0$. As explained in \cite[Proposition 3.3]{si}, the combinatorics of infinite nested sums imply that there is a surjective algebra morphism
\begin{align} \label{SZ stuffle}
(\QB^0,\ast_{\operatorname{SZ}})&\longrightarrow \nonformal, \\
b_{s_1}\cdots b_{s_l}&\longmapsto \SZ(s_1,\ldots,s_l).
\end{align}
For simplicity, we will also write $\SZ(b_{s_1}\cdots b_{s_l})$ for $\SZ(s_1,\ldots,s_l)$.\\
\\
A second set of relations between Schlesinger-Zudilin multiple $q$-zeta values comes from an involution on $\QB^0$.
\begin{defi}
The involution $\tau:\QB^0\to\QB^0$ is the $\QQ$-linear map defined by
\[\tau(b_{k_1}b_0^{m_1}\cdots b_{k_d}b_0^{m_d})=b_{m_d+1}b_0^{k_d-1}\cdots b_{m_1+1}b_0^{k_1-1}\]
for all $k_1,\ldots,k_d\geq1,\ m_1,\ldots,m_d\geq0$.
\end{defi} \noindent
The Schlesinger-Zudilin multiple $q$-zeta values are invariant under this involution (\cite[Theorem 4]{ta}),
\begin{align} \label{SZ tau invariance}
\SZ(\tau(w))=\SZ(w), \qquad w\in \QB^0.
\end{align} 
These relations are homogeneous in weight. To prove the equality in \eqref{SZ tau invariance}, we first observe that for each $k\geq1$
\begin{align*}
\frac{X^k}{(1-X)^k}=\sum_{r>0} \binom{r-1}{k-1} X^r,
\end{align*}
and that for each $m\geq0$
\begin{align*}
\sum_{N_1>n_1>\dots>n_m>N_2} 1=\binom{N_1-N_2-1}{m}, \qquad N_1,N_2\geq1.
\end{align*}
Thus, we obtain for $k_1,\ldots,k_d\geq1,\ m_1,\ldots,m_d\geq0$
\begin{align} \label{SZ binomial coefficients}
\SZ(k_1,\{0\}^{m_1},\ldots,k_d,\{0\}^{m_d})&=\sum_{\substack{N_1>\dots> N_d>N_{d+1}=0 \\ N_i>n_1^{(i)}>\dots >n_{m_i}^{(i)}>N_{i+1} \\ \text{for i=1,\ldots,d}}} \frac{q^{N_1k_1}}{(1-q^{N_1})^{k_1}}\cdots \frac{q^{N_dk_d}}{(1-q^{N_d})^{k_d}} \\
&=\sum_{\substack{N_1>\dots>N_d>N_{d+1}=0,\\ r_1,\ldots,r_d>0}} \prod_{i=1}^{d}\binom{N_i-N_{i+1}-1}{m_i}\binom{r_i-1}{k_i-1}  q^{N_ir_i}. 
\end{align}
By substituting the variables $N_i-N_{i+1}=r_{d+1-i}$, we obtain that this sum is equal to
\begin{align} \label{SZ variable substitution}
&=\sum_{\substack{N_1>\dots>N_d>N_{d+1}=0,\\ r_1,\ldots,r_d>0}} \prod_{i=1}^{d}\binom{r_{d+1-i}-1}{m_i}\binom{N_{d+1-i}-N_{d+2-i}-1}{k_i-1}  q^{N_ir_i} \\
&=\SZ(m_d+1,\{0\}^{k_d-1},\ldots,m_1+1,\{0\}^{k_1-1}).
\end{align}
Let $\lambda=((u_1,\ldots,u_d),(v_1,\ldots,v_d))$ be a partition of some positive integer $N$. This means that the ordered numbers $u_1>\dots>u_d>0$ are the different parts of the partition $\lambda$ and the numbers $v_1,\ldots,v_d>0$ are their multiplicities,
\[N=u_1v_1+\dots+u_dv_d.\] 
We refer to the number $d$ as the length of the partition $\lambda$. Then the calculation in \eqref{SZ binomial coefficients} shows that the coefficient of $q^N$ in any Schlesinger-Zudilin multiple $q$-zeta value is a sum over all partitions of $N$ of length $d$. In this picture, the variable substitution in \eqref{SZ variable substitution} corresponds to the conjugation of partitions.
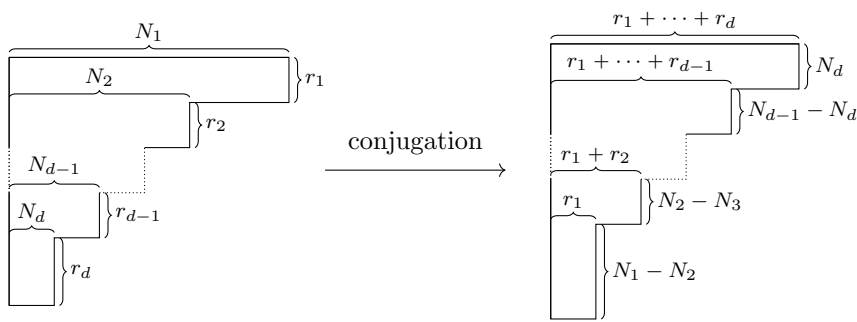
\[\begin{tikzpicture}[scale=0.6]
\draw (0,2.5) -- (0,0) -- (1,0) -- (1,1.5) -- (2,1.5) -- (2,2.5);
\draw [densely dotted] (0,2.5) -- (0,3.5);
\draw [densely dotted] (2,2.5) -- (3,2.5) -- (3,3.5);
\draw (3,3.5) -- (4,3.5) -- (4,4.5) -- (6.2,4.5) -- (6.2,5.5) -- (0,5.5) -- (0,3.5);
\draw (0,5.5) -- (6.2,5.5); 
\draw (0,5.5) -- (0,3.5); 
\draw [densely dotted] (0,2.5) -- (0,3.5);
\draw (0,0) -- (0,2.5);
\draw[decoration={brace,raise=2pt},decorate] (0,5.5) -- node[above=2pt] {\tiny $N_1$} (6.2,5.5);
\draw[decoration={brace,raise=2pt},decorate] (0,4.5) -- node[above=2pt] {\tiny $N_2$} (4,4.5);
\draw[decoration={brace,raise=2pt},decorate] (0,2.5) -- node[above=2pt] {\tiny $N_{d-1}$} (2,2.5);
\draw[decoration={brace,raise=2pt},decorate] (0,1.5) -- node[above=2pt] {\tiny $N_d$} (1,1.5);
\draw[decoration={brace,raise=2pt},decorate] (6.2,5.5) -- node[right=2pt] {\tiny $r_1$} (6.2,4.5);
\draw[decoration={brace,raise=2pt},decorate] (4,4.5) -- node[right=2pt] {\tiny $r_2$} (4,3.5);
\draw[decoration={brace,raise=2pt},decorate] (2,2.5) -- node[right=2pt] {\tiny $r_{d-1}$} (2,1.5);
\draw[decoration={brace,raise=2pt},decorate] (1,1.5) -- node[right=2pt] {\tiny $r_d$} (1,0);
	
\draw [->] (7,3) -- node[above=2pt]{\footnotesize{conjugation}} (11,3);
	
\begin{scope}[shift={(12,0.3)}]	
\draw (0,2.5) -- (0,-0.6) -- (1,-0.6) -- (1,1.5) -- (2,1.5) -- (2,2.5);
\draw [densely dotted] (0,2.5) -- (0,3.5);
\draw [densely dotted] (2,2.5) -- (3,2.5) -- (3,3.5);
\draw (3,3.5) -- (4,3.5) -- (4,4.5) -- (5.5,4.5) -- (5.5,5.5) -- (0,5.5) -- (0,3.5);		
\draw (0,5.5) -- (5.5,5.5); 
\draw (0,5.5) -- (0,3.5); 
\draw [densely dotted] (0,2.5) -- (0,3.5);
\draw (0,-0.6) -- (0,2.5);		
\draw[decoration={brace,raise=2pt},decorate] (0,5.5) -- node[above=2pt] {\tiny $r_1 + \dots + r_d$} (5.5,5.5);
\draw[decoration={brace,raise=2pt},decorate] (0,4.5) -- node[above=2pt] {\tiny $r_1 + \dots + r_{d-1}$} (4,4.5);
\draw[decoration={brace,raise=2pt},decorate] (0,2.5) -- node[above=2pt] {\tiny $r_1+r_2$} (2,2.5);
\draw[decoration={brace,raise=2pt},decorate] (0,1.5) -- node[above=2pt] {\tiny $r_1$} (1,1.5);
\draw[decoration={brace,raise=2pt},decorate] (5.5,5.5) -- node[right=2pt] {\tiny $N_d$} (5.5,4.5);
\draw[decoration={brace,raise=2pt},decorate] (4,4.5) -- node[right=2pt] {\tiny $N_{d-1}-N_d$} (4,3.5);
\draw[decoration={brace,raise=2pt},decorate] (2,2.5) -- node[right=2pt] {\tiny $N_2-N_3$} (2,1.5);
\draw[decoration={brace,raise=2pt},decorate] (1,1.5) -- node[right=2pt] {\tiny $N_1-N_2$} (1,-0.6);
\end{scope}
\end{tikzpicture} \]
\begin{rem} 
Any polynomial $f\in \QQ[X_1,\ldots,X_d,Y_1,\ldots,Y_d]$ can be evaluated at the partition $\lambda=((u_1,\ldots,u_d),(v_1,\ldots,v_d))$ as
\[f(\lambda)=f(u_1,\ldots,u_d,v_1,\ldots,v_d).\]
Denote by $\mathcal{P}_d(N)$ the set of all partitions of $N$ of length $d$. Then we associate to a polynomial $f\in \QQ[X_1,\ldots,X_d,Y_1,\ldots,Y_d]$ the family of generating series 
\[\operatorname{Gen}_f^{(d)}(q)=\sum_{N\geq1}\left(\sum_{\lambda\in \mathcal{P}_d(N)} f(\lambda)\right)q^N, \qquad d\geq0.\]
The only partition of length $0$ is the empty partition, so $\operatorname{Gen}_f^{(0)}(q)=1$. 
Since the polynomial ring $\QQ[X]$ is spanned by the binomial coefficients $\binom{X}{k}=\prod_{j=1}^k \frac{X+1-j}{j}$, $k\geq0$, the computations in \eqref{SZ binomial coefficients} imply
\[\nonformal=\spanQ\{\operatorname{Gen}_f^{(d)}(q)\mid d\geq0,\ f\in\QQ[X_1,\ldots,X_d,Y_1,\ldots,Y_d]\}. \]
More generally, in \cite[cf (1.6)]{bi} it is shown that the space $\nonformal$ is exactly the image of the polynomial functions on partitions under the $q$-bracket.
\end{rem} 
\begin{rem}
The algebra $\nonformal$ unifies many models for $q$-analogs of multiple zeta values occurring in the literature. In \cite{bk}, the authors consider for integers $s_1\geq1,\ s_2,\ldots,s_l\geq0$, and polynomials $R_1,\ldots,R_l\in \QQ[t]$, with $R_1(0)=0$, $q$-analogues of the form
\begin{align} \label{generic qMZV} \zq(s_1,\ldots,s_l;R_1,\ldots,R_l)=\sum_{n_1>\dots >n_l>0} \frac{R_1(q^{n_1})}{(1-q^{n_1})^{s_1}}\cdots \frac{R_l(q^{n_l})}{(1-q^{n_l})^{s_l}}\in \QQ\llbracket q\rrbracket.
\end{align}
One verifies that
\[\nonformal=\spanQ\{\zq(s_1,\ldots,s_l;R_1,\ldots,R_l)\mid l\geq0, s_1\geq1,s_2,\ldots,s_l\geq0,\deg(R_j)\leq s_j\}.\]
Choosing $R_i=t^{s_i}$ for $i=1,\ldots,l$, one recovers the Schlesinger-Zudilin multiple $q$-zeta values. The choice $R_i=t^{s_i-1}$ for $i=1,\ldots,l$ yields the Bradley-Zhao multiple $q$-zeta values (\cite{bra}, \cite{zh2})
\begin{align*} 
\zq^{\operatorname{BZ}}(s_1,\ldots,s_l)=\sum_{n_1>\dots> n_l>0} \frac{q^{n_1(s_1-1)}}{(1-q^{n_1})^{s_1}}\cdots \frac{q^{n_l(s_l-1)}}{(1-q^{n_l})^{s_l}} \qquad (s_1\geq2,s_2,\ldots,s_l\geq1).
\end{align*}
There are several more models for multiple $q$-zeta values proposed by Bachmann (\cite{ba}), Ohno-Okuda-Zudilin (\cite{ooz}), and Okounkov (\cite{ok}), which also possess an expression in terms of the generic multiple $q$-zeta values in \eqref{generic qMZV}. A detailed overview is given in \cite{bri}. 
\end{rem}

\subsection{Balanced multiple $q$-zeta values.} In analogy to the case of multiple zeta values and in order to apply similar techniques as Racinet in his thesis, we are interested in a spanning set of $\nonformal$ satisfying weight-homogeneous relations. In \cite{bu2} a spanning set is given, which satisfies conjecturally exactly the relations between Schlesinger-Zudilin multiple $q$-zeta values modulo lower weight. We explain these in the following.
\begin{defi} \label{bst algebra} We define the balanced quasi-shuffle product $\bst$ on $\QB$ recursively by $\1\bst w=w\bst \1=w$ and
\[b_iu\bst b_jv=b_i(u\bst b_jv)+b_j(b_iu\bst v)+\begin{cases} b_{i+j}(u\bst v), \quad & i,j \geq1, \\ 0 & \text{else}. \end{cases} \]
\end{defi} \noindent
The algebra $(\QB,\bst)$ is exactly the associated graded algebra to $(\QB,\ast_{\operatorname{SZ}})$ with respect to the weight
\[\wt(b_{s_1}\cdots b_{s_l})=s_1+\dots+s_l+\#\{i\mid s_i=0\},\qquad s_1,\ldots,s_l\geq0.\]
Moreover, the balanced quasi-shuffle product is a natural combination of the stuffle product \eqref{stuffle mzv} and the shuffle product \eqref{shuffle mzv} of multiple zeta values.
\begin{thm} \label{thm balanced} (\cite[Theorem 10.4]{bu2} There is a surjective, $\tau$-invariant algebra morphism
\begin{align*}
(\QB^0,\bst)&\longrightarrow \nonformal, \\
b_{s_1}\cdots b_{s_l}&\longmapsto\zq(s_1,\ldots,s_l).
\end{align*}
\end{thm} \noindent
The elements $\zq(s_1,\ldots,s_l)$, $s_1\geq1,\ s_2,\ldots,s_l\geq0$ are called \emph{balanced multiple $q$-zeta values}.
\vspace{0,2cm} \\
For each $k\geq2$ even, the element $\zq(k)$ equals the classical Eisenstein series of weight $k$ with rational Fourier coefficients. \\
In general the explicit construction of the balanced multiple $q$-zeta values it quite involved, therefore we omit the construction here and refer the interested reader to \cite{bu2}. For the purposes of this paper, it suffices to define the balanced multiple $q$-zeta values as the images of the algebra morphism in Theorem \ref{thm balanced}.
\begin{con} \label{all relations balanced} All algebraic relations in the algebra $\nonformal$ are a consequence of the balanced quasi-shuffle product formula and the $\tau$-invariance of the balanced multiple $q$-zeta values.
\end{con} \noindent
Similar to the case of multiple zeta values (Conjecture \ref{Conj all relations in Z}), this conjecture will motivate the definition of the algebra $\formal$ given in Section \ref{formal}.

\subsection{Multiple Eisenstein series.} 
Let $\mathbb{H}=\{z\in \mathbb{C}\mid \operatorname{Im}(z)>0\}$ be the upper half plane. For each $k\geq2$ even, the classical Eisenstein series of weight $k$ are given (up to some normalization) by
\begin{align*}
\mathbb{G}_k(z)=\sum_{\substack{m>0 \\ \vee (m=0 \wedge n>0)}} \frac{1}{(mz+n)^k}=\zeta(k)+\frac{(-2\pi i)^k}{(k-1)!}\sum_{m,d>0} d^{k-1}q^{md} \qquad (q=e^{2\pi i z},\ z\in \mathbb{H}).
\end{align*}
To obtain a multiple version of the Eisenstein series, define an order on the lattice $\ZZ z+\ZZ$ (where $z\in \mathbb{H}$) as follows: We have $m_1z+n_1\succ m_2z+n_2$ if and only if $m_1>m_2$ or $(m_1=m_2 \wedge n_1>n_2)$. 
\begin{defi} \label{def MES} (\cite{gkz}) To integers $k_1\geq3, k_2,\ldots,k_d\geq2$ associate the \emph{multiple Eisenstein series}
\[\mathbb{G}_{k_1,\ldots,k_d}(z)=\sum_{\substack{\lambda_1\succ\cdots\succ \lambda_d\succ 0 \\ \lambda_i\in \ZZ z+\ZZ}} \frac{1}{\lambda_1^{k_1}\cdots \lambda_d^{k_d}} \qquad (z\in \mathbb{H}).\]
\end{defi} \noindent
One observes immediately that multiple Eisenstein series are holomorphic functions on the upper half plane $\mathbb{H}$ and that they satisfy the stuffle product formula. Moreover, the multiple Eisenstein series possess a Fourier expansion, where the constant term is given by multiple zeta values. The second building block of the Fourier expansion are the \emph{brackets} (introduced in Bachmann's master thesis and studied in \cite{bk2})
\begin{align} \label{def brackets}
\hspace{-0,7cm}g(k_1,\ldots,k_d)=\sum_{\substack{u_1>\dots>u_d>0 \\ v_1,\ldots,v_d>0}} \frac{v_1^{k_1-1}}{(k_1-1)!}\cdots \frac{v_d^{k_d-1}}{(k_d-1)!}q^{u_1v_1+\dots+u_dv_d}\in \nonformal, \ k_1,\ldots,k_d\geq1.
\end{align}
\begin{prop} \label{Fourier expansion MES} (\cite[Theorem 1.4]{ba2}) There are $\alpha_{l_1,\ldots,l_d,j}^{k_1,\ldots,k_d}\in \ZZ$, such that we have for all $k_1\geq3, k_2,\ldots,k_d\geq2$
\[\mathbb{G}_{k_1,\ldots,k_d}(z)=\zeta(k_1,\ldots,k_d)+\sum_{\substack{0< j< d \\ l_1+\dots+l_d=k_1+\dots+k_d \\ l_1\geq2,l_2,\ldots,l_d\geq1}} \alpha_{l_1,\ldots,l_d,j}^{k_1,\ldots,k_d}\ \zeta(l_1,\ldots,l_j)\hat{g}(l_{j+1},\ldots,l_d)+\hat{g}(k_1,\ldots,k_d),\]
where $\hat{g}(k_1,\ldots,k_d)=(-2\pi i)^{k_1+\dots+k_d}g(k_1,\ldots,k_d)\in \nonformal[\pi i]$.
\end{prop} \noindent
A very natural question is how to determine the algebraic relations between multiple Eisenstein series. Since the multiple Eisenstein series satisfy the stuffle product formula and in their Fourier expansion the multiple zeta values occur, one could expect that they also satisfy a variant or subset of the extended double shuffle relations.

\subsection{Combinatorial multiple Eisenstein series.} 
To attack the above question, we will restrict to $q$-series with rational coefficients. Precisely, this means that we take a rational solution to the extended double shuffle equations $\beta(k_1,\ldots,k_d)$, $k_1\geq2,k_2,\ldots,k_d\geq1$, instead of the multiple zeta values. To obtain for $d=1$ the classical Eisenstein series of weight $k$ with rational coefficients, we additionally fix
\begin{align*}
\beta(k)=\frac{\zeta(k)}{(2\pi i)^k}=-\frac{B_k}{2k!},\quad k\text{ even},\qquad \qquad \beta(k)=0, \quad k\text{ odd}.
\end{align*}
Moreover, we will use the $q$-series $g(k_1,\ldots,k_d)\in\nonformal$ given in \eqref{def brackets} instead of the series $\hat{g}(k_1,\ldots,k_r)\in\nonformal[\pi i]$. Combining these rational solutions $\beta(k_1,\ldots,k_d)$ and a bi-version of the brackets $g(k_1,\ldots,k_d)$ inspired by the Fourier expansion of the multiple Eisenstein series (Proposition \ref{Fourier expansion MES}), one obtains the following.
\begin{thm} \label{thm  CMES} (\cite{bb}) Let $\Ybi=\{y_{k,m}\mid k\geq1,m\geq0\}$ be an alphabet and $(\QYbi,\ast)$ be the stuffle algebra, i.e., the product $\ast$ on $\QYbi$ is given by $\1\ast w=w\ast\1=w$ and 
\[y_{k_1,m_1}u\ast y_{k_2,m_2}v=y_{k_1,m_1}(u\ast y_{k_2,m_2}v)+y_{k_2,m_2}(y_{k_1,m_1}u\ast v)+y_{k_1+k_2,m_1+m_2}(u\ast v)\]
for all $u,v,w\in \QYbi$. There exists a surjective algebra morphism
\begin{align*}
(\QYbi,\ast)&\longrightarrow \nonformal, \\
y_{k_1,m_1}\cdots y_{k_d,m_d}&\longmapsto G\binom{k_1,\ldots,k_d}{m_1,\ldots,m_d}.
\end{align*}
Additionally, the generating series
\[\mathfrak{G}\binom{X_1,\ldots,X_d}{Y_1,\ldots,Y_d}=\sum_{\substack{k_1,\ldots,k_d\geq 1 \\ m_1,\ldots,m_d\geq0}} G\binom{k_1,\ldots,k_d}{m_1,\ldots,m_d}X_1^{k_1-1}\frac{Y_1^{m_1}}{m_1!}\cdots X_d^{k_d-1}\frac{Y_d^{m_d}}{m_d!}\]
are swap invariant, i.e., we have for all $d\geq1$
\[\mathfrak{G}\binom{X_1,\ldots,X_d}{Y_1,\ldots,Y_d}=\mathfrak{G}\binom{Y_1+\dots+Y_d,\ldots,Y_1+Y_2,Y_1}{X_d,X_{d-1}-X_d,\ldots, X_1-X_2}.\]
\end{thm} \noindent
We call the elements $G\binom{k_1,\ldots,k_d}{m_1,\ldots,m_d}$, $k_1,\ldots,k_d\geq1,\ m_1,\ldots,m_d\geq0$ the \emph{combinatorial bi-multiple Eisenstein series}. We set for all $k_1,\ldots,k_d\geq1$
\begin{align*}
G(k_1,\ldots,k_d)=G\binom{k_1,\ldots,k_d}{0,\ldots,0},
\end{align*}
and call those elements \emph{combinatorial multiple Eisenstein series}. 
\vspace{0,2cm}\\
The explicit construction of the combinatorial bi-multiple Eisenstein series is given in \cite{bb}. For the following, it suffices to define them as the images of the algebra morphism given in Theorem \ref{thm CMES}.
\vspace{0,2cm} \\
Combinatorial multiple Eisenstein series should be seen as the rational version of multiple Eisenstein series given in Definition \ref{def MES}. So, we expect them to satisfy the same relations as (regularized) multiple Eisenstein series.
\begin{ex} For $k>m\geq0$, we obtain
\begin{align*}
G\binom{k}{m}&=-\delta_{m,0}\frac{B_k}{2k!}-\delta_{k,1}\frac{B_{m+1}}{2(m+1)}+\frac{1}{(k-1)!}\sum_{u,v\geq1}u^mv^{k-1}q^{uv} \\
&=\frac{(k-m-1)!}{(k-1)!}\left(q\frac{\diff}{\diff q}\right)^mG(k-d).
\end{align*} 
\end{ex} \noindent
So these bi-indices essentially mean that we also include partial derivatives of the combinatorial multiple Eisenstein series. This allows us to find a variant of the double shuffle relations for the combinatorial bi-multiple Eisenstein series. 
\begin{ex} Combining the product formula and the swap invariance from Theorem \ref{thm  CMES} we obtain for $k_1,k_2\geq1$
\begin{align*}
G(k_1)G(k_2)&=G(k_1,k_2)+G(k_2,k_1)+G(k_1+k_2) \\
&=\sum_{j=1}^{k_1+k_2-1} \left(\binom{j-1}{k_1-1}+\binom{j-1}{k_2-1}\right) G(j,k_1+k_2-j)\\
&\hspace{0,8cm}+\binom{k_1+k_2-2}{k_1-1}G\binom{k_1+k_2-1}{1}.
\end{align*}
\end{ex} \noindent
Conjecturally, these bi-indices do not yield any new elements and are just a nice tool to describe the relations and product formulas of the combinatorial multiple Eisenstein series. Every combinatorial bi-multiple Eisenstein series should be a $\QQ$-linear combination of combinatorial multiple Eisenstein series. In other words, we expect the following.
\begin{con} \label{conj G^{f,0} equals G^f}There is a surjective algebra morphism
\begin{align*}
(\QY,\ast)&\longrightarrow \nonformal,\\ 
y_{k_1}\cdots y_{k_d}&\longmapsto G(k_1,\ldots k_d).
\end{align*}
\end{con} \noindent
This is equivalent to Bachmann's conjecture that brackets and bi-brackets span the same space (see also \cite[Remark 6.16]{bb}). 
\vspace{0,2cm} \\
We expect that all algebraic relations between combinatorial (bi-)multiple Eisenstein series, and hence also between multiple Eisenstein series are a consequence of the stuffle product formula and the swap invariance in Theorem \ref{thm  CMES}. Under the isomorphism $\varphi_{\#}$ given in \cite[Theorem 7.10]{bu2}, these two kinds of relations translate into the balanced quasi-shuffle product and $\tau$-invariance. Therefore, the algebra $\formal$ defined in the following Section \ref{formal} is also a formalization of combinatorial (bi-)multiple Eisenstein series.
\begin{rem}
(i) A formal version of the combinatorial bi-multiple Eisenstein series of depth $\leq2$ was already studied in detail in \cite{bkm}. 
\vspace{0,1cm} \\
(ii) In \cite{bim}, they also study a formalization of the combinatorial bi-multiple Eisenstein series. This means, up to the isomorphism  $\varphi_{\#}$, they also study the algebra $\formal$. Their focus lies on derivatives on the algebra $\formal$. Similar to the case of quasi-modular forms, they obtain an $\mathfrak{sl}_2$-action on $\formal$.
\end{rem} 

\section{The algebra $\formal$} \label{formal}

\noindent
We introduce the algebra $\formal$ and give some of its basic properties. Then, we present the realization of $\formal$ into the algebra $\nonformal$ obtained from the balanced multiple $q$-zeta values. Recall that $(\QB,\bst)$ denotes the balanced quasi-shuffle algebra given in Definition \ref{bst algebra}.
\begin{defi} \label{def formal qMZV}
We set
\[\formal=\faktor{(\QB,\bst)}{\rel},\]
where $\rel$ is the ideal in $(\QB,\bst)$ generated by the set $\{b_0\}\cup\{w-\tau(w)\mid w\in\QB^0\}$.
\end{defi} \noindent
Denote by $\symb(w)$ the class of an element $w\in \QB$ in the quotient space $\formal$ and set $\symb(\1)=1$. Then $\formal$ is the algebra spanned by the elements $\symb(w)$, $w\in \B^*$, which exactly satisfy the following relations
\renewcommand{\arraystretch}{1,1}
\begin{center} \centering \begin{tabular}{cclcl}		
(i) & $\symb(b_0)$ & $=$ & $0$, & \\ 
(ii) & $\symb(v\bst w)$ & $=$ & $\symb(v)\symb(w)$, & $v,w\in\QB$, \\
(iii) & $\symb(\tau(w))$ & $=$ & $\symb(w)$, & $w\in \QB^0$. \\
\end{tabular} \end{center} \renewcommand{\arraystretch}{1}
In (ii) the expression $\symb(v)\symb(w)$ denotes the product induced by $\bst$ in the quotient algebra $\formal$. The algebra $\formal$ inherits the notion of weight and depth from the algebra $\QB$, i.e., we set
\begin{align*}
\wt(\symb(b_{s_1}\cdots b_{s_l}))=s_1+\dots+s_l+\#\{i\mid s_i=0\}, \qquad \dep(\symb(b_{s_1}\cdots b_{s_l}))=l-\#\{i\mid s_i=0\}.
\end{align*}
Since both the balanced quasi-shuffle product and the involution $\tau$ are homogeneous in weight, the algebra $\formal$ is graded by weight.
\vspace{0,2cm} \\
By definition, the algebra $\formal$ is equipped with a universal property. For every $\QQ$-algebra $R$ and every algebra morphism 
\[\varphi:(\QB,\bst)\longrightarrow R,\] 
which is $\tau$-invariant on $\QB^0$ and satisfies $\varphi(b_0)=0$, there exists a unique algebra morphism $\widetilde{\varphi}:\formal\to  R$, such that the following diagram commutes
\begin{equation} \label{universal property} 
\begin{tikzcd} \QB \arrow[drr, "\varphi"] \arrow[rr,"\symb"] && \formal \arrow[d,dashrightarrow,"\widetilde{\varphi}"] \\ &&  R.
\end{tikzcd}
\end{equation} 
\begin{defi} A \emph{realization} of the algebra $\formal$ is a pair $(R,\varphi)$, where $R$ is a $\QQ$-algebra and $\varphi:\formal\to R$ is a surjective algebra morphism into $R$.
\end{defi}
\begin{thm} \label{realization of formal} There exists a realization of $\formal$ into $\nonformal$ given by
\begin{align*} 
\formal&\longtwoheadrightarrow \nonformal, \\
\symb(w)&\longmapsto\zqreg(w).
\end{align*}
\end{thm} \noindent
The elements $\zqreg(w)$ are the regularized balanced multiple $q$-zeta values and will be explained below.
\vspace{0,2cm} \\
As a reformulation of Conjecture \ref{all relations balanced}, we expect the map in Theorem \ref{realization of formal} to be an isomorphism of weight-graded algebras. In particular, the algebra $\formal$ should determine all algebraic relations between multiple $q$-zeta values and multiple Eisenstein series.
\begin{proof}
By Theorem \ref{thm balanced}, there is a $\tau$-invariant algebra morphism
\begin{align*}
\zq:(\QB^0,\bst)&\longrightarrow \nonformal, \\
b_{s_1}\cdots b_{s_l}&\longmapsto \zq(s_1,\ldots,s_l).
\end{align*}
In \cite[Proposition 6.2]{bu2} a regularization map is given. Precisely, we have an algebra isomorphism with respect to the balanced quasi-shuffle product
\begin{align*}
\regT:\QB^0[T]&\longrightarrow \QB, \\
wT^n&\longmapsto w\bst b_0^{\bst n}.
\end{align*}
Applying the inverse $\regT^{-1}:\QB\to \QB^0[T]$ and then evaluating in $T=0$ yields the regularization morphism
\begin{align} \label{regularization morphism}
\reg:(\QB,\bst)\longrightarrow(\QB^0,\bst).
\end{align}
By construction, the restriction of $\reg$ to $\QB^0$ is just the identity. The \emph{regularized balanced multiple $q$-zeta values} are given by 
\[\zq^{\reg}(w)=\zq(\reg(w)),\qquad w\in \QB.\]
A consequence of Theorem \ref{thm balanced} and the regularization process is that there is a surjective algebra morphism
\begin{align*}
(\QB,\bst)&\longrightarrow \nonformal, \\
w&\longmapsto\zq^{\reg}(w),
\end{align*}
which is $\tau$-invariant on $\QB^0$ and satisfies $\zq^{\reg}(b_0)=0$. Applying the universal property \eqref{universal property} to the map $\zq^{\reg}:\QB\to \nonformal$, we obtain the desired realization of $\formal$ in $\nonformal$.
\end{proof} \noindent
For each $k\geq2$, the element $\zq(k)$ is the classical Eisenstein series $G(k)$ of weight $k$ with rational Fourier coefficients. In particular, the preimages $\symb(b_2),\ \symb(b_4)$, and $\symb(b_6)$ under the map $\formal\twoheadrightarrow \nonformal$ in Theorem \ref{realization of formal} must be algebraically independent. We deduce that there is an algebra isomorphism
\begin{align} \label{formal quasimod}
\QQ[\symb(b_2),\symb(b_4),\symb(b_6)]\simeq \quasimod,
\end{align}
where $\quasimod=\QQ[G(2),G(4),G(6)]$ denotes the algebra of quasi-modular forms with rational coefficients. Thus, one should view the elements in $\QQ[\symb(b_2),\symb(b_4),\symb(b_6)]$ as formal quasi-modular forms. A more structural description for the formal (quasi-)modular forms in terms of derivatives will be given in \cite{bim}.

\section{The balanced quasi-shuffle Hopf algebra} \label{balanced quasi-shuffle Hopf algebra}

\noindent
To give the affine scheme corresponding to the algebra $\formal$, we first explain the balanced quasi-shuffle algebra and determine its completed dual. We show that these Hopf algebras are a natural combination of the shuffle Hopf algebra $(\QX,\shuffle,\dec)$ and the stuffle Hopf algebra $(\QY,\ast,\dec)$ and their duals considered in Section \ref{MZV}. 
\vspace{0,2cm} \\
Let $\dec$ be the deconcatenation coproduct on $\QB$. From \cite[Theorem 3.1, 3.2]{h}, we immediately obtain the following.
\begin{prop} \label{QB Hopf algebra} The tuple $(\QB,\bst,\dec)$ is a weight-graded, commutative Hopf algebra.
\end{prop} \noindent
For any commutative $\QQ$-algebra $R$ with unit, let $\RBc{R}$ be the free algebra of non-commutative power series in the alphabet $\B$ with coefficients in $R$.
\begin{defi} Define the coproduct $\bco:\RBc{R}\to\RBc{R}\otimes\RBc{R}$ by
\begin{align*} 
\bco(b_i)&=\1\otimes b_i+b_i\otimes \1+\sum_{j=1}^{i-1} b_j\otimes b_{i-j}, \qquad i\geq 0,
\end{align*}
and extend this with respect to the concatenation product.
\end{defi}
\begin{prop} \label{dual Hopf algebra} The tuple $(\RBc{ R},\conc,\bco)$ is a complete, cocommutative Hopf algebra. The pairing
\begin{align*}
\RBc{ R}\otimes \QB&\longrightarrow R, \\
\Phi\otimes w&\longmapsto(\Phi\mid w),	
\end{align*}
where $(\Phi\mid w)$ denotes the coefficient of $w$ in $\Phi$, gives a duality between the weight-graded Hopf algebra $(\QB,\bst,\dec)$ and the complete Hopf algebra $(\RBc{ R},\conc,\bco)$.
\end{prop}
\begin{proof} We prove the duality of $(\QB,\bst,\dec)$ and $(\RBc{R},\conc,\bco)$ with respect to the given pairing. Then it is an immediate consequence that $(\RBc{R},\conc,\bco)$ is a cocommutative Hopf algebra. It is well-known that $\dec$ and $\conc$ are dual maps. For $u,v\in \QB$ one obtains
\begin{align*} 
\big(\bco(b_i)\mid u\otimes v\big)&=\Big(\1\otimes b_i+b_i\otimes\1+\sum_{j=1}^{i-1} b_i\otimes b_{i-j}\ \Big|\ u\otimes v\Big)=\big(b_i\mid u\bst v\big) 
\end{align*}
The last equality holds, since the word $b_i$ appears in the product $u\bst v$ if and only if $u=\1,\ v=b_i$ or $u=b_i,\ v=\1$ or $u=b_{j},\ v=b_{i-j}$ for some $j=1,\ldots,i-1$.
Since $\bco$ is compatible with the concatenation product by definition and the letters $b_i$ generate the algebra $(\RBc{ R},\conc$), we deduce that the maps $\bst$ and $\bco$ are dual.
\end{proof} \noindent
The antipode $S:\RBc{R}\to \RBc{R}$ of the Hopf algebra $(\RBc{R},\conc,\bco)$ is the anti-automorphism given by
\begin{align*} 
S(b_0)&=-b_0,  \\
S(b_a)&=\sum_{r=1}^a\sum_{\substack{j_1+\dots+j_r=a \\ j_1,\ldots,j_r\geq1}}(-1)^rb_{j_1}\cdots b_{j_r}, \qquad a\geq 1.
\end{align*}
We end this section by explaining how the balanced quasi-shuffle Hopf algebra combines the shuffle Hopf algebra $(\QX,\shuffle,\dec)$ and the stuffle Hopf algebra $(\QY,\ast,\dec)$ from Section \ref{MZV}. Straight-forward computations show the following.
\begin{prop} There are two surjective Hopf algebra morphisms
\begin{align*}
\left(\QB,\bst, \dec\right)&\longtwoheadrightarrow (\QX,\shuffle,\dec), \\ 
b_i&\longmapsto 
\begin{cases} 
x_i, \quad &i\in\{0,1\}, \\
0, \quad &i\geq2,
\end{cases}
\end{align*}
and
\begin{align*}
(\QB,\bst, \dec)&\longtwoheadrightarrow (\QY,\ast,\dec), \\ 
b_i&\longmapsto 
\begin{cases}
0, \quad &i=0, \\
y_i, \quad &i\geq1.
\end{cases}
\end{align*}
\end{prop} \noindent
By duality, we also obtain two injective Hopf algebra morphisms
\begin{align}  \label{def theta_X}
\theta_\X:(\RXc{R},\conc,\co)&\longhookrightarrow (\RBc{R},\conc,\bco),\\ 
x_i&\longmapsto b_i, \quad i\in\{0,1\}, \nonumber
\end{align}
and 
\begin{align} \label{def theta_Y}
\theta_\Y:(\RYc{R},\conc,\stco)&\longhookrightarrow (\RBc{R},\conc,\bco),\\ 
y_i&\longmapsto b_i,\quad i\geq1. \nonumber
\end{align}
\begin{rem} \label{shuffle, stuffle and q-stuffe sequence} The stuffle Hopf algebra $(\QY,\ast,\dec)$ can be identified with the Hopf subalgebra of $(\QB,\bst,\dec)$ spanned by all words which do not contain the letter $b_0$. This leads to an injective Hopf algebra morphism
\begin{align*}
(\QY,\ast,\dec)&\longhookrightarrow	(\QB,\bst, \dec) , \\ 
y_i&\longmapsto b_i, \quad i\geq 1.
\end{align*}
On the other hand, we have $b_1\bst b_1=2b_1^2+b_2$, and thus the words containing only the letters $b_0$ and $b_1$ do not span a Hopf subalgebra of $(\QB,\bst,\dec)$. So, the shuffle Hopf algebra $(\QX,\shuffle,\dec)$ does not canonically embed into $(\QB,\bst,\dec)$. One obtains a sequence of Hopf algebras
\begin{align} \label{short sequence} 
0\longrightarrow (\QY,\ast,\dec)\longhookrightarrow (\QB,\bst,\dec)\longtwoheadrightarrow (\QX,\shuffle,\dec)\longrightarrow 0,
\end{align}
which is nearly exact (the only exceptions are the span of the words $b_1^n$, $n\geq 0$). 
\end{rem}

\section{The affine scheme $\BM$} \label{affine scheme BM}

\noindent
Similar to the case of formal multiple zeta values (Section \ref{MZV}), we assign to the algebra $\formal$ an affine scheme $\BM$. This affine scheme has values in the complete Hopf algebra $(\RBc{R},\conc,\bco)$ presented in Section \ref{balanced quasi-shuffle Hopf algebra}. 
\begin{defi} \label{def BM}
For each commutative $\QQ$-algebra $R$ with unit, denote by $\BM(R)$ the set of all non-commutative power series $\Phi$ in $\RBc{R}$ satisfying
\renewcommand{\arraystretch}{1,1}
\begin{center} \centering \begin{tabular}{cclc}
(i) & $(\Phi|b_0)$ & $=$ & $0$, \\ 
(ii) & $\bco(\Phi)$ & $=$ & $\Phi \otimes \Phi$, \\
(iii) & $\tau(\Pi_0(\Phi))$ & $=$ & $\Pi_0(\Phi)$. \\
\end{tabular} \end{center} \renewcommand{\arraystretch}{1}
Here $\Pi_0$ denotes the $R$-linear extension of the canonical projection $\QB \to \QB^0$, which is the identity on $\QB^0$ and maps all words starting with $b_0$ to $0$.
\vspace{0,2cm} \\
For all $\lambda,\mu,\nu\in R$, let $\BM_{(\lambda,\mu,\nu)}(R)$ be the subset of all $\Phi\in\BM(R)$ additionally satisfying
\[\hspace{3,6cm}\text{(iv)} \quad (\Phi\mid b_2)=\lambda,\quad (\Phi\mid b_4)=\mu,\quad (\Phi\mid b_6)=\nu.\]
We abbreviate $\BMo(R)=\BM_{(0,0,0)}(R)$.
\end{defi} \noindent
Condition (iv) is motivated by the observation that $\QQ[\symb(b_2),\symb(b_4),\symb(b_6)]$ is isomorphic to the algebra of quasi-modular forms (see \eqref{formal quasimod}). So one might expect that for any arbitrary choice of $\lambda,\mu,\nu$ the set $\BM_{\lambda,\mu,\nu}(R)$ is non-empty.
\begin{thm} \label{BM affine scheme} For every commutative $\QQ$-algebra $R$ with unit and $\lambda,\mu,\nu\in R$, there are bijections
\begin{align*} 
\BM(R)&\simeq \Hom(\formal,R),\\ \BM_{(\lambda,\mu,\nu)}(R)&\simeq\Hom\Big(\faktor{\formal\!}{\!\left( \symb(b_2)-\lambda,\ \symb(b_4)-\mu,\ \symb(b_6)-\nu\right)},R\Big).
\end{align*}
\end{thm} \noindent
In particular, $\BM:\QQ\operatorname{-Alg}\to \operatorname{Sets}$ is an affine scheme represented by the algebra $\formal$ and $\BM_{(\lambda,\mu,\nu)}:\QQ\operatorname{-Alg}\to\operatorname{Sets}$ is an affine scheme represented by 
\[\faktor{\formal\!}{\!\left( \symb(b_2)-\lambda,\ \symb(b_4)-\mu,\ \symb(b_6)-\nu\right)}.\]
\begin{proof}
The first bijection is given by 
\begin{align*}
\chi:\Hom(\formal,R)&\longrightarrow \BM(R), \\
\varphi&\longmapsto \sum_{w\in\B^*} \varphi(\symb(w))w.
\end{align*} 
Let $\varphi:\formal\to R$ be a $\QQ$-algebra morphism. Since $\symb(b_0)=0$, we obtain \[(\chi(\varphi)|b_0)=0.\] 
The product in $\formal$ is induced by the balanced quasi-shuffle product $\bst$, thus we have \[(\chi(\varphi)|u\bst v)=(\chi(\varphi)|u)(\chi(\varphi)|v)\qquad  \text{for all } u,v\in \QB.\] 
From the duality of $\bst$ and $\bco$  with respect to the pairing $(-\mid -)$ (Proposition \ref{dual Hopf algebra}), we deduce 
\[(\bco(\chi(\varphi))\mid u\otimes v)=(\chi(\varphi)\mid u \bst v)=(\chi(\varphi)\mid u)\ (\chi(\varphi)\mid v)=(\chi(\varphi)\otimes \chi(\varphi)\mid u\otimes v)\]
for all $u,v\in \QB$. In particular, the power series $\chi(\varphi)$ is grouplike for $\bco$. Since $\tau$ maps words onto words, the $\tau$-invariance of the elements $\symb(w)$ for $w\in\QB^0$ implies \[\tau\big(\Pi_0(\chi(\varphi))\big)=\Pi_0(\chi(\varphi)).\] This shows that $\chi(\varphi)$ is contained in the set $\BM(R)$ and therefore the map $\chi$ is well-defined.
The inverse of $\chi$ is given by
\begin{align*}
\BM(R)&\longrightarrow \Hom(\formal,R),\\
\Phi&\longmapsto\Big(\symb(w)\mapsto (\Phi\mid w)\Big).
\end{align*}
It is an immediate consequence that $\chi$ also induces a bijection
\begin{align*}
\Big\{\varphi\in\Hom(\formal,R)\ \Big|\  \varphi(\symb(b_2))=\lambda,\ \varphi(\symb(b_4))=\mu,\ \varphi(\symb(b_6))=\nu\Big\} &\longrightarrow \BM_{(\lambda,\mu,\nu)}(R), \\
\varphi &\longmapsto f(\varphi).
\end{align*}
By the universal property of quotient spaces, the set on the left-hand side is in bijection to
\[\Hom\Big(\faktor{\formal\!}{\!\left( \symb(b_2)-\lambda,\ \symb(b_4)-\mu\ ,\symb(b_6)-\nu\right)},R\Big).\] 
\end{proof} \noindent
Applying the bijection in Theorem \ref{BM affine scheme} to the $\QQ$-algebra morphism $\formal\to \nonformal,\ \symb(w)\mapsto \zq^{\reg}(w)$ given in Theorem \ref{realization of formal} shows that
\[\sum_{w\in\B^*} \zq^{\reg}(w)w\in \BM_{G(2),G(4),G(6)}(\nonformal).\]
As before, $G(2),\ G(4),\ G(6)$ denote the classical Eisenstein series of weight $2,4,6$ with rational Fourier coefficients.

\section{Relation of the affine schemes $\DM$ and $\BM$ and consequences} \label{BM and DM}

\noindent
We show that the affine scheme $\DM$ assigned to formal multiple zeta values (Definition \ref{def DM}) embeds into the affine scheme $\BM$ (Definition \ref{def BM}). Therefore, we obtain a surjective algebra morphism from the algebra $\formal$ into the algebra $\Zf$.
\vspace{0,2cm} \\
Let $R$ be a commutative $\QQ$-algebra with unit. To relate the sets $\DM(R)$ and $\BM(R)$, we need the embeddings of the dual shuffle and stuffle Hopf algebra into $(\RBc{ R},\conc,\bco)$, those were defined in \eqref{def theta_X} and \eqref{def theta_Y} as
\begin{align*}
\theta_\X:(\RXc{ R},\conc,\co)&\longrightarrow (\RBc{ R},\conc,\bco),\ x_i\longmapsto b_i \quad (i\in\{0,1\}) \\ 
\theta_\Y:\hspace{0,15cm} (\RYc{ R},\conc,\stco)&\longrightarrow (\RBc{ R},\conc,\bco),\ y_i\longmapsto b_i \quad (i\geq 1).
\end{align*} 
To capture the fact that the map $\tau$ is an anti-morphism, we consider the following Hopf algebra anti morphism
\begin{align*} \label{def theta_X anti}
\theta_\X^{\anti}: (\RXc{ R},\conc,\co)&\longrightarrow (\RBc{ R},\conc,\bco),\\ x_{\varepsilon_1}\cdots x_{\varepsilon_n}&\longmapsto b_{\varepsilon_n}\cdots b_{\varepsilon_1}. \nonumber
\end{align*}
\begin{lem} \label{lem embedding DM BM} For the canonical projections $\Pi_0:\RBc{ R}\to \RBc{ R}^0$ (Definition \ref{def BM}) and $\Pi_\Y:\RXc{ R}\to\RYc{ R}$ (Definition \ref{def DM}), we have
\[\tau\circ\Pi_0\circ\theta_\X^{\anti}=\theta_\Y\circ\Pi_\Y.\]
\end{lem}
\begin{proof} For a word $w=x_0^{k_1-1}x_1\cdots x_0^{k_d-1}x_1$ in $\RXc{ R}$ (where $k_1,\ldots,k_d\geq1$), we compute
\begin{align*}
(\tau\circ\Pi_0\circ \theta_\X^{\anti})(w)&=\tau(b_1b_0^{k_d-1}\cdots b_1b_0^{k_1-1})=b_{k_1}\cdots b_{k_d}=\theta_\Y(y_{k_1}\cdots y_{k_d})=(\theta_\Y\circ\Pi_\Y)(w).
\end{align*} 
On the other hand, if $w=vx_0$ for some word $v$ in $\RXc{R}$, we obtain
\begin{align*}
(\tau\circ\Pi_0\circ \theta_\X^{\anti})(w)=(\tau\circ\Pi_0)(b_0\theta_\X^{\anti}(v))=0=\Pi_\Y(vx_0)=(\theta_\Y\circ \Pi_\Y)(w)
\end{align*}
\end{proof}
\begin{thm} \label{DM into BM} For each commutative $\QQ$-algebra $R$ with unit, we have an injective map
\begin{align*} 
\theta: \DM(R)&\longhookrightarrow \BM(R), \\
\phi&\longmapsto\theta_\X^{\anti}(\phi)\theta_\Y(\phi_\ast),
\end{align*}
where we denote (as in Definition \ref{def DM}) \[\phi_\ast=\phi_{\operatorname{corr}}\Pi_\Y(\phi)=\exp\left(\sum\limits_{n\geq2}\frac{(-1)^{n-1}}{n}(\Pi_\Y(\phi)|y_n)y_1^n\right)\Pi_\Y(\phi)\in \RYc{ R}.\]
\end{thm}  \noindent
The chosen order of the factors in the definition of $\theta$ is necessary for the compatibility of the projections $\Pi_\Y$ and $\Pi_0$ under the map $\theta$.
\begin{proof} Let $\phi\in \DM( R)$. We have $(\phi\mid x_0)=0$ and hence $(\theta(\phi)\mid b_0)=0$. Since $\theta_\X^{\anti}$, $\theta_\Y$ are coalgebra morphisms and $\phi$ and $\phi_\ast$ are grouplike for $\co$ and $\stco$, we compute
\begin{align*}
\bco\big(\theta(\phi)\big)&=\bco\big(\theta_\X^{\anti}(\phi)\big)\bco\big(\theta_\Y(\phi_\ast)\big)=\big(\theta_\X^{\anti}(\phi) \otimes \theta_\X^{\anti}(\phi)\big)\big(\theta_\Y(\phi_\ast) \otimes \theta_\Y(\phi_\ast)\big) \\
&=\theta(\phi) \otimes \theta(\phi).
\end{align*}
By Lemma \ref{lem embedding DM BM}, we obtain
\begin{align*}
\tau\Big(\Pi_0(\theta(\phi))\Big)&=\tau\Big(\Pi_0(\theta_\X^{\anti}(\phi))\theta_\Y(\phi_\ast)\Big)=\tau\Big(\theta_\Y(\phi_\ast)\Big)\tau\Big(\Pi_0(\theta_\X^{\anti}(\phi))\Big)\\
&=\tau\Big(\theta_\Y(\phi_\ast)\Big)\theta_\Y\Big(\Pi_\Y(\phi)\Big)=\tau\Big(\theta_\Y(\Pi_\Y(\phi))\Big)\tau\Big(\theta_\Y(\phi_{\operatorname{corr}})\Big)\theta_\Y\Big(\Pi_\Y(\phi)\Big) \\
&=\Pi_0\Big(\theta_\X^{\anti}(\phi)\Big)\theta_\Y\Big(\phi_{\operatorname{corr}}\Big)\theta_\Y\Big(\Pi_\Y(\phi)\Big)=\Pi_0\Big(\theta_\X^{\anti}(\phi)\Big)\theta_\Y(\phi_\ast) \\
&=\Pi_0\Big(\theta(\phi)\Big).
\end{align*}
Note that $\theta_\Y(\phi_{\operatorname{corr}})$ consists of the letter $b_1$ and is therefore $\tau$-invariant. We proved that $\theta(\phi)$ is an element in $\BM( R)$ and thus the map $\theta$ is well-defined. \\
Next, we show injectivity. The elements $\phi\in \DM( R)$ satisfy $(\phi\mid x_1)=0$ and hence also $(\phi\mid x_1^n)=0$ for all $n\geq1$. Thus, any non-trivial word in $\theta_\X^{\anti}(\phi)$ contains the letter $b_0$ and every non-trivial word in $\theta_\Y(\phi_\ast)$ contains a letter $b_i$, $i>1$. As $(\theta_\X^{\anti}(\phi)\mid\1)=(\theta_\Y(\phi)\mid \1)=1$, we deduce
\[\theta(\phi)=\theta_\X^{\anti}(\phi)\theta_\Y(\phi)=\theta_\X^{\anti}(\phi)+\theta_\Y(\phi)+\begin{pmatrix} \text{linear combinations of words containing} \\ \text{the letters } b_0 \text{ and } b_i \text{ for some } i>1\end{pmatrix}.\]
In particular, the part of $\theta(\phi)$ consisting of the letters $b_0,b_1$ is exactly $\theta_\X^{\anti}(\phi)$. Therefore, the injectivity of $\theta_\X^{\anti}$ implies the injectivity of $\theta$.
\end{proof} 
\begin{rem}
(i) Since the set $\DM(R)$ is non-empty for any commutative $\QQ$-algebra $ R$ with unit, the existence of the injective map in Theorem \ref{DM into BM} shows that $\BM(R)$ is non-empty.
\vspace{0,2cm} \\
(ii) By applying the defining conditions of $\DM_{\lambda}(R)$, one obtains that each $\phi\in \DM_{\lambda}(R)$ satisfies
\[(\phi\mid x_0^3x_1)=(\phi_\ast\mid y_4)=\frac{2}{5}\lambda^2,\qquad (\phi\mid x_0^5x_1)=(\phi_\ast\mid y_6)=\frac{8}{35}\lambda^3.\]
These numbers come from Euler's formula for the even zeta values, precisely one computes $\frac{\zeta(4)}{\zeta(2)^2}=\frac{2}{5}$ and $\frac{\zeta(6)}{\zeta(2)^3}=\frac{8}{35}$. Therefore, the map $\theta$ in Theorem \ref{DM into BM} restricts to an embedding
\[\theta:\DM_{\lambda}(R)\longhookrightarrow \BM_{\big(\lambda,\frac{2}{5}\lambda^2,\frac{8}{35}\lambda^3\big)}(R).\]
\end{rem} \noindent
Since $\theta:\DM\to \BM$ is an injective natural transformation of affine schemes, we obtain by applying Yoneda's Lemma a surjective morphism between the representing algebras.
\begin{thm} \label{surjection G^f to Z^f} There is a surjective algebra morphism 
\begin{align*} 
p:\hspace{0,4cm}\formal\ &\longrightarrow\hspace{1,3cm} \Zf,\\
\symb(w)\ &\longmapsto\sum_{\substack{w=uv \\ u\in\{b_0,b_1\}^*,\ v\in\{b_i\mid i\geq1\}^*}}\zeta^f_{\shuffle}\left((\theta_\X^{\anti})^{-1}(u)\right)\zeta^f_\ast\left(\theta_\Y^{-1}(v)\right) \qquad (w\in \B^*).
\end{align*}
\end{thm} \noindent
Here $\zeta^f_\ast(u)$ denotes the stuffle regularized formal multiple zeta values and $\zeta^f_{\shuffle}(v)$ denotes the shuffle regularized formal multiple zeta values.
\vspace{0,1cm} \\
In other words, Theorem \ref{surjection G^f to Z^f} shows that there is a realization of the algebra $\formal$ into the algebra $\Z$ of multiple zeta values, given by the composition of the maps $p:\formal\to \Zf$ and $\Zf\to\Z,\ \zeta_{\shuffle}^f(w)\mapsto\zeta_{\shuffle}(w)$.
\begin{proof} The element $\id_{\Zf}\in \Hom(\Zf,\Zf)$ corresponds to the element $\sum\limits_{w\in \X^*} \zeta^f_{\shuffle}(w)w$ in $\DM(\Zf)$ under the bijection given in \eqref{bijections DM}. We obtain
\begin{align*}
\theta\Bigg(\sum_{w\in \X^*} \zeta^f_{\shuffle}(w) w\Bigg)&=\theta_\X^{\anti}\Bigg(\sum_{u\in \X^*} \zeta^f_{\shuffle}(u)u\Bigg)\theta_\Y\Bigg(\sum_{v\in \Y^*} \zeta^f_\ast(v)v\Bigg) \\
&= \sum_{u\in \X^*,\ v\in \Y^*} \zeta^f_{\shuffle}(u)\zeta^f_\ast(v)\theta_\X^{\anti}(u)\theta_\Y(v) \\
&= \sum_{u\in \{b_0,b_1\}^*,\ v\in \{b_i\mid i\geq 1\}^*} \zeta^f_{\shuffle}\Big((\theta_\X^{\anti})^{-1}(u)\Big)\zeta^f_\ast\Big(\theta_\Y^{-1}(v)\Big)uv.
\end{align*}
So under the bijection given in Theorem \ref{BM affine scheme} the element $\theta\Big(\sum\limits_{w\in \X^*} \zeta^f_{\shuffle}(w) w\Big)\in \BM(\Zf)$ corresponds to the algebra morphism
\begin{align*}
p:\hspace{0,4cm} \formal\ &\longrightarrow\hspace{1,2cm} \Zf, \\
\symb(w)\ &\longmapsto \sum_{\substack{w=uv \\ u\in\{b_0,b_1\}^*,\ v\in\{b_i\mid i\geq1\}^*}} \zeta^f_{\shuffle}\left((\theta_\X^{\anti})^{-1}(u)\right)\zeta^f_\ast\left(\theta_\Y^{-1}(v)\right) \qquad (w\in\B^*).
\end{align*}
By Yoneda's Lemma, this is exactly the algebra morphism induced by the natural transformation $\theta:\DM\to \BM$ of affine schemes.
\end{proof} \noindent
The map $p:\formal\to \Zf$ can be seen as a formal limit $q\to1$. For example, one computes
\[p(\symb(b_2b_3))=\zeta_{\shuffle}^f(\1)\zeta_\ast^f(y_2y_3)=\zeta^f(2,3),\]
and similarly
\[\lim_{q\to1} (1-q)^5\zq(2,3)=\zeta(2,3).\]
In \cite[Theorem 4.18]{bi} it is proven in a slightly different context that this holds in general. Interpreting the algebra $\formal$ as a formalization of multiple Eisenstein series (see Section \ref{nonformal}), the map $p:\formal\to \Zf$ can also be seen as a formal version of taking the constant term.

\noindent\vspace{-0,5cm} \\
\paragraph{\textbf{The algebra $\formal$ and (extended) double shuffle relations.}} We indicate how to obtain a variant of the double shuffle relations in the algebra $\formal$, which reduce to the double shuffle relations under the surjective algebra morphism $p:\formal\to \Zf$. \\
On the one hand, we can multiply the elements $\symb(w)$ in the algebra $\formal$ with respect to the balanced quasi-shuffle product $\bst$. On the other hand, we can first apply the $\tau$-invariance to both factors, then multiply with respect to the balanced quasi-shuffle product, and finally again apply the $\tau$-invariance to all terms. 
\begin{ex} The previous described procedure yields for all $k_1,k_2\geq1,\ m_1,m_2\geq0$
\begin{align*}
&\symb(b_{k_1}b_0^{m_1})\symb(b_{k_2}b_0^{m_2})\\ 
&=\sum_{j=0}^{m_1+m_2} \left(\binom{m_1+m_2-j}{m_2}\symb(b_{k_1}b_0^{j}b_{k_2}b_0^{m_1+m_2-j})+\binom{m_1+m_2-j}{m_1}\symb(b_{k_2}b_0^jb_{k_1}b_0^{m_1+m_2-j})\right)\\
&\hspace{0,7cm}+\binom{m_1+m_2}{m_1}\symb(b_{k_1+k_2}b_0^{m_1+m_2}) \\
&=\sum_{j=1}^{k_1+k_2-1} \left(\binom{j-1}{k_1-1}\symb(b_jb_0^{m_1}b_{k_1+k_2-j}b_0^{m_2})+\binom{j-1}{k_2-1}\symb(b_jb_0^{m_2}b_{k_1+k_2-j}b_0^{m_1})\right) \\
&\hspace{0,7cm}+\binom{k_1+k_2-2}{k_1-1}\symb(b_{k_1+k_2-1}b_0^{m_1+m_2+1}).
\end{align*}
In particular, we get in the case $m_1=m_2=0$ that
\begin{align} \label{eq1}
\symb(b_{k_1})\symb(b_{k_2})&=\symb(b_{k_1}b_{k_2})+\symb(b_{k_2}b_{k_1})+\symb(b_{k_1+k_2}) \\
&=\sum_{j=1}^{k_1+k_2-1} \left(\binom{j-1}{k_1-1}+\binom{j-1}{k_2-1}\right)\symb(b_jb_{k_1+k_2-j})+\binom{k_1+k_2-2}{k_1-1}\symb(b_{k_1+k_2-1}b_0).
\end{align}
Applying the algebra morphism $p:\formal\to\Zf$ from Theorem \ref{surjection G^f to Z^f} yields for $k_1,k_2\geq2$
\begin{align*}
\zeta^f(k_1)\zeta^f(k_2)&=\zeta^f(k_1,k_2)+\zeta^f(k_2,k_1)+\zeta^f(k_1+k_2) \\
&=\sum_{j=2}^{k_1+k_2-1} \left(\binom{j-1}{k_1-1}+\binom{j-1}{k_2-1}\right)\zeta^f(j,k_1+k_2-j).
\end{align*}
So we recover the double shuffle relations in depth $2$. We also obtain relations between formal multiple zeta values coming from regularization. For example, \eqref{eq1} in the case $k_1=1,\ k_2=2$ reads
\begin{align*}
\symb(b_1)\symb(b_2)&=\symb(b_1b_2)+\symb(b_2b_1)+\symb(b_3) \\
&=\symb(b_1b_2)+2\symb(b_2b_1)+\symb(b_2b_0),
\end{align*}
and applying the surjection $p:\formal\to\Zf$ gives
\begin{align*}
\Big(\zeta_{\shuffle}^f(1)+\zeta_\ast^f(1)\Big)\zeta_\ast^f(2)&=\zeta_\ast^f(1,2)+\zeta_{\shuffle}^f(1)\zeta_\ast^f(2)+\zeta_\ast^f(2,1)+\zeta_\ast^f(3) \\
&=\zeta_\ast^f(1,2)+\zeta_{\shuffle}^f(1)\zeta_\ast^f(2)+2\zeta_\ast^f(2,1)+0.
\end{align*}
Then cancellation yields Euler's famous relation $\zeta^f(3)=\zeta^f(2,1)$.
\vspace{0,1cm} \\
There are several examples in the literature of how to deduce relations between multiple zeta values from the relations in $\nonformal$ (or more generally, from relations between $q$-analogs of multiple zeta values), see for example \cite[Section 7.6]{bk2}, \cite{bra}, or \cite[Section 3.4]{si}. The approach given here allows a very algebraic view on the extended double shuffle relations of multiple zeta values. 
\end{ex} 

\noindent\vspace{-0,5cm} \\
\paragraph{\textbf{Outlook.}} By linearizing the defining equations of $\BMo$ one obtains a vector space $\bmo$. Since we expect $\BMo$ to be an affine group scheme, the linearized space $\bmo$ should be equipped with a Lie algebra structure. In \cite{bu}, the $q$-Ihara bracket is introduced, which is experimentally shown to preserve the space $\bmo$ in small weights. In particular, the $q$-Ihara bracket should give rise to the group multiplication on $\BMo$ and an exponential map $\exp:\widehat{\bmo}\to\BMo$. Consequences of this would be that the algebra $\formal$ is a free polynomial algebra and that the quotient $\formal/(\symb(b_2),\symb(b_4),\symb(b_6))$ is a Hopf algebra, where the coproduct should be seen as a generalization of Goncharov's coproduct. This is illustrated in more detail in \cite{bu} and will be part of forthcoming work.

\end{document}